\documentclass[11pt]{article}
\usepackage{mathrsfs}
\usepackage{amsmath}
\usepackage{amsfonts}
\usepackage{graphicx}
\usepackage{enumerate}
\usepackage{setspace}
\usepackage{color}
\usepackage{bbm}
\usepackage{amsthm}
\usepackage{amssymb}

\newtheorem{theorem}{Theorem}
\newtheorem{proposition1}{Proposition}
\newtheorem{proposition*}{Proposition}
\newtheorem{corollary1}{Corollary}
\newtheorem{lemma1}{Lemma}
\newtheorem{lemma*}{Lemma}
\newtheorem{remark1}{Remark}

\begin{document}
\title{A maximum principle for progressive optimal control of mean-filed forward-backward stochastic system involving random jumps and impulse controls}
\author{Tian Chen \hspace{1cm} Kai Du \hspace{1cm} Zongyuan Huang$^*$ \hspace{1cm} Zhen Wu}

\maketitle

\vspace{2mm}\noindent\textbf{Abstract.}\quad In this paper, we study an optimal control problem of a mean-field forward-backward stochastic system with random jumps in progressive structure, where both regular and singular controls are considered in our formula. In virtue of the variational technology, the related stochastic maximum principle (SMP) has been obtained, and it is essentially different from that in the classical predictable structure. Specifically, there are three parts in our SMP, i.e. continuous part, jump part and impulse part, and they are respectively used to characterize the characteristics of the optimal controls at continuous time, jump time and impulse time. This shows that the progressive structure can more accurately describe the characteristics of the optimal control at the jump time. We also give two linear-quadratic (LQ) examples to show the significance of our results.

\vspace{2mm} \noindent\textbf{Keywords.}\quad Mean-field Control, Progressive Structure, Random Jumps, Singular Control, Maximum Principle.

\footnotetext[1]{School of Mathematics, Shandong University, Jinan, Shandong, 250100, China. e-mail: huangzy@sdu.edu.cn.}
\footnotetext[2]{This work was supported by Natural Science Foundation of China, Grant/Award Number: 11831010, 11901353, 61961160732 and 12001319; Natural Science Foundation of Shandong Province, Grant/Award Number: ZR2019ZD42 and ZR2020QA025; the Taishan Scholars Climbing Program of Shandong, Grant/Award Number: TSPD20210302; the Postdoctoral Innovative Talent Support Program, Grant/Award Number: BX20200199.}


\section{Introduction}

Nowadays, mean-field control and game problems have attracted and gained more and more researchers' attentions, because of the extensive applications and significant theoretical values in many fields, such as economics , information technology, engineering, and system control \cite{Bensoussan2013,Espinosa2015,Huang2015,Li2016,wang2021social,hu2022decentralized,Du2022,XW2023,TSSA2023,XS2023}.  There are two typically used problem frames and models for the mean-field control. The first one is the large population system when contains many cooperative agents, and it is also called social optimal control firstly studied by \cite{Huang2012,Nourian2012}. The another one is mean-field type control system, where the dynamic is given by some mean-field system, likes mean-field stochastic differential equation (MF-SDE). Actually, since the theory of MF-SDEs studied by \cite{Kac1956}, the related mean-field type control has been a fascinating field of study, and a number of infusive and classic results have emerged. For example, Yong \cite{Yong2013} studied the linear-quadratic optimal control problem in mean-field system and got the related feedback form optimal control; the maximum principle of the mean field stochastic control problems with the joint distribution of the controlled state and the control process was gotten by \cite{Acciaio2019}; the mean-field optimal control in backward stochastic differential system was studied by Li, Sun and Xiong \cite{LSX2019}. Since the mean-field control is a very fast growing research field, it is very hard to exhaustively and roundly account all the developments of it,  and here we refer to \cite{Bensoussan2013,Carmona2018,Wang20202,SY2020} for interesting readers to get more details about it.

The stochastic maximum principle (SMP) is one of the important tools to solve the optimal control problem. Promoting the development of the maximum principle is not only of great significance in theory, but also of great practical significance in the application of finance, engineering and so on (see \cite{Pucci2007,Weitzman2009,Chen2010,Boltyanski2011}). It was first formulated by Pontryagin in 1950s and it converted the optimization problems into maximizing the corresponding Hamiltonian functions. Bismut \cite{Bismut1978} first introduced the linear backward stochastic differential equation (BSDE) as the the adjoint equations and Peng \cite{Peng1990} first obtained the general stochastic maximum principle by introducing the general BSDE as the second-order adjoint equations. Tang and Li \cite{Tang1994} obtained the stochastic maximum principle for the optimal control problem of stochastic system with random jumps, where both diffusion and jump terms explicitly depend on control variables. Framstad, \O ksendal and Sulem \cite{Framstad2004} obtained the sufficient maximum principle for the optimal control of jump diffusions and give the applications to finance. Wu  \cite{Wu2013} obtained general maximum principle for optimal control of forward-backward systems, where control domains are non-convex and forward diffusion coefficients explicitly depend on control variables. Song, Tang and Wu \cite{Song2020} obtained the stochastic maximum principle for progressive optimal control of the stochastic system with random jumps by introducing a new method of spike variation. They derived the rigorous maximum principle and revealed the essential difference between the stochastic control system with jump and the system without jump. Chen and Wu \cite{Chen2022} obtained the stochastic maximum principle for progressive optimal control of partially observed mean-field stochastic system with random jumps.

In recently years, stochastic impulse control problems have been widely concerned because of its wide application. Davis and Norman \cite{Davis1990} and \O ksendal and Sulem \cite{Oksendal2002} studied the portfolio selection problem with transaction costs by impulse control. Wu and Zhang \cite{Wu2011} first studied the stochastic maximum principle for optimal impulse controls. They obtained the necessary and sufficient maximum principle of forward-backward systems involving impulse controls, where the control domain is required to be convex.

Inspired by the above research, in this paper, we would like to study a new type of optimal control problem with mean-filed forward-backward stochastic system involving random jumps and impulse controls, and under the new progressive control framework, the necessary and sufficient stochastic maximum principle for the progressive optimal control is obtained. Since there are essential difference between the stochastic system with jump and the system without jump, it is essential to accurately characterize the characteristics of optimal control at the time of jump. In order to solve this problem, we introduced two mutually singular measures $\mu$ and $\nu$ to separate the SMP of continuous time and jump time, and obtained accurate characterizations of the optimal control at continuous time, jump time, and impulse time. Our results are very different from existing works.  The innovation of this paper is divided into the following two points. The first point is that the disturbance coefficients of the convex variation of general control and impulse control do not need to be the same. As far as we know, the coefficient is required to be the same in all existing works. But we don't think this assumption is necessary. The second point is that the main results of our paper, the stochastic maximum principle, which is divided into three parts. The first is the continuous part, which is used to describe the characteristics of optimal control when the jump does not occur. The second part is the jump part, which is used to describe the characteristics of optimal control when the jump occurs. The last part is the impulse part, which is used to describe the characteristics of optimal impulse control. In the progressive structure, the description of the optimal control of the stochastic system with jump is more accurate.

The main contributions of this paper are as follows:
\begin{itemize}
  \item (i) The stochastic maximum principle of mean-field forward-backward stochastic system involving random jumps and impulse controls in progressive framework is given. It is different from what we already know, there are three separate parts used to characterize optimal control in continuous, jump and impulse cases in our SMP.
  \item (ii) Some estimates about the mean-field SDE with jump and impulse in progressive structure are given, which plays an important role in proving SMP.
  \item (iii) As an application, the related linear-quadratic (LQ) optimal control problem in MF-FBSDE with random jump and impulse control variable is considered. Moreover, Under some basic assumptions, we achieve its unique optimal control. Furthermore, the optimal control and corresponding state feedback representation is given by decoupling methods.
\end{itemize}

The rest of this paper is organized as follows. In section 2, we give some preliminary results of progressive structure. Section 3 formulates the optimal control problem and gives some assumptions. In section 4, we employ the convex variation and give some estimations. To get the necessary and sufficient maximum principle in progressive structure, the related adjoint equation is given by section 5. In section 6, we give two LQ examples to illustrate the significants of the progressive structure. We also give the optimal control and corresponding state feedback representation by decoupling methods.

\section{Preliminaries}
We fix $T>0$ as fixed time horizon. Let $\left(\Omega, \mathscr{F}, \{\mathscr{F}_t\}_{t\geq 0}, \mathbb P\right)$ be a given filtered complete probability space. On this space, there is a 1-dimensional $\mathscr{F}_t$-Brownian motion $B_t$ and a Poisson random measure $N$ on $[0,T]\times \mathcal{E}$ adapted to $\mathscr{F}_t$, where $\mathcal{E}$ is a standard measure space with a $\sigma$-field $\mathscr{B}(\mathcal{E})$. We assume that the Brownian motion and Poisson random measure are independent. The mean measure of $N$ is a measure on $([0,T]\times \mathcal{E}, \mathscr{B}([0,T])\otimes \mathscr{B}(\mathcal{E}))$ which has the form $Leb \times \lambda$, where $Leb$ denotes the Lebesgue measure on $[0,T]$ and $\lambda$ is a finite measure on $\mathcal{E}$. For any $D\in \mathscr{B} (\mathcal{E})$ and $t\in [0,T]$, since $\lambda(D)<\infty$, we set $\tilde{N}(\omega, [0,t]\times D)\triangleq N(\omega, [0,t]\times D)-t\lambda(D)$. It is well known that $\tilde{N}(\omega, [0,t]\times D)$ is a martingale for every $D$. We assume that $\{\mathscr{F}_t\}_{t\geq 0}$ is generated $B_t$ and $N$, that is,
\begin{equation*}
  \mathscr{F}_t \triangleq\sigma \left(N([0,s],A), 0\leq s\leq t, A\in \mathscr{B}(\mathcal{E})\right)\vee \sigma \left(B_s; 0\leq s\leq t\right)\vee \mathscr{N},
\end{equation*}
where $\mathscr{N}$ denotes the totality of $\mathbb{P}$-null sets. Then $\mathscr{F}_t$ satisfies the usual conditions. Let $\{\tau_i\}$ be a given sequence of increasing $\mathscr{F}_t$ stopping times such that $\tau_i\uparrow+\infty$. We denote by $\mathscr{K}$ the class of processes $\eta_{\cdot}=\sum_{i\geq 1} \eta_i\mathbbm{1}_{[\tau_i,T]}(\cdot)$ such that each $\eta_i$ is a $R$-valued $\mathscr{F}_{\tau_i}$ measurable random variable. It is worth noting that, the assumption $\tau_i\uparrow+\infty$ implies that at most finitely many impulses may occurs on $[0,T]$.

Suppose that $M$ is a Euclid space, and $\mathscr{B}(M)$ is the Borel $\sigma$-field on $M$. A process $X: [0,T]\times \Omega \rightarrow M$ is called progressive (predictable) if $X$ is $\mathscr{G}/ \mathscr{B}(M) (\mathscr{P}/\mathscr{B}(M))$ measurable, where $\mathscr{G}(\mathscr{P})$ is the progressive (predictable) $\sigma$-field on $[0,T]\times \Omega$; a process $X: [0,T]\times \Omega \times \mathcal{E} \rightarrow M$ is called $\mathcal{E}$-progressive ($\mathcal{E}$-predictable) if $X$ is $\mathscr{G}\otimes \mathscr{B}(\mathcal{E})/ \mathscr{B}(M) (\mathscr{P} \otimes \mathscr{B}(\mathcal{E})/\mathscr{B}(M))$ measurable. In this paper, the stochastic integral we used is more general, that is, the integrand of the stochastic integral is $\mathcal{E}$-progressive rather than $\mathcal{E}$-predictable.

Now we introduce some notations. Given a process $X_t$ with c\`adl\`ag paths, $X_{0-}\triangleq0$ and $\Delta X_t\triangleq X_t- X_{t-}, t\geq 0$, let $\mu$ denote the measure on $\mathscr{F}\otimes \mathscr{B}([0,T])\otimes \mathscr{B} (\mathcal{E})$ generated by $N$ that $\mu(A)=E[\int_0^T\int_{\mathcal{E}} I_{A} N(\mathrm{d}s,\mathrm{d}e)]$. For any $\mathscr{F}\otimes \mathscr{B} ([0,T])\otimes \mathscr{B}(\mathcal{E})/\mathcal{B}(R)$ measurable integrable process $X$, we set $\mathbb{E}[X]\triangleq\int X \mathrm{d}\mu$ and denote by $\check{\mathbb{E}}[X]\triangleq\mathbb{E}[X|\mathscr{P}\otimes \mathscr{B}(\mathcal{E})]$ the Radon-Nikodym derivatives with respect to $\mathscr{P}\otimes \mathscr{B}(\mathcal{E})$. Note that $\mathbb{E}$ is not an expectation (for $\mu$ is not a probability measure), though it has similar properties to expectation. For convenience, we assume that $\tilde{x}$ denotes the ``expected value" of the state variable $x$ throughout this paper. More details of the stochastic integral of random measure are in Song, Tang and Wu \cite{Song2020}.

Then we give several necessary propositions as follows. The proof we can see \cite{Song2020}, so we omit it.
\begin{proposition1}
  If $H$ is a $\mathcal{E}$-progressive process such that $E[\int_0^T\int_{\mathcal{E}} H N(\mathrm{d}t,\mathrm{d}e)]<\infty$, then
  \begin{equation*}
    \left(\int_0^{\cdot}\int_{\mathcal{E}}HN(\mathrm{d}t,\mathrm{d}e)\right)^p_t=\int_0^t\int_{\mathcal{E}} \check{\mathbb{E}}[H] \lambda(\mathrm{d}e)\mathrm{d}s,
  \end{equation*}
  where $X^p$ is the dual predictable projection of $X$.
\end{proposition1}

\begin{proposition1}\label{2.2}
  If $H$ is ${\mathcal{E}}$-progressive and satisfies $E[\int_0^T\int_{{\mathcal{E}}}H^2 N(\mathrm{d}t,\mathrm{d}e)]<\infty$, then we have
  \begin{equation*}
    \begin{aligned}
      \int_0^T\int_{\mathcal{E}} H \tilde{N}(\mathrm{d}t,\mathrm{d}e)&=\int_0^T\int_{\mathcal{E}} H N(\mathrm{d}t,\mathrm{d}e) - \int_0^T\int_{\mathcal{E}} \check{\mathbb{E}}[H]\lambda(\mathrm{d}e)\mathrm{d}t,\\
 \Delta(H.\tilde{N})_t&=\int_{\mathcal{E}} H N({t},\mathrm{d}e),\\
      [H.\tilde{N},H.\tilde{N}]_t&=\int_0^t\int_{\mathcal{E}} H^2 N(\mathrm{d}t,\mathrm{d}e).
    \end{aligned}
  \end{equation*}
\end{proposition1}

\begin{remark1}\label{225}
  Under the conditions of the proposition above, we have
  \begin{equation}\label{2226}
    E\left[\int_0^T\int_{\mathcal{E}} H N(\mathrm{d}t,\mathrm{d}e)\right] =E\left[\int_0^T\int_{\mathcal{E}} \check{\mathbb{E}}[H]\lambda(\mathrm{d}e)\mathrm{d}t \right].
  \end{equation}
  In particular, if $H$ is ${\mathcal{E}}$-predictable, we have the well-known result
  \begin{equation*}
    E\left[\int_0^T\int_{\mathcal{E}} H N(\mathrm{d}t,\mathrm{d}e)\right] =E\left[\int_0^T\int_{\mathcal{E}} H\lambda(\mathrm{d}e)\mathrm{d}t \right].
  \end{equation*}
\end{remark1}

Now, we consider the following standard mean-field SDE with jump and impulses in progressive structure (MF-SDEP, for short):
\begin{equation}\label{SDE}
  \begin{aligned}
    X_t=x_0&+\int_0^t\int_{\mathcal{E}} b(s,X_s,E[X_s],e)\lambda(\mathrm{d}e)\mathrm{d}s+\int_0^t\int_{\mathcal{E}} \sigma(s,X_s,E[X_s],e)\lambda(\mathrm{d}e)\mathrm{d}B_s+\int_0^t G_s\mathrm{d}\eta_s \\
    &+\int_0^t \int_{\mathcal{E}} \gamma(s,X_{s-},E[X_{s-}],e)N(\mathrm{d}s,\mathrm{d}e) +\int_0^t\int_{\mathcal{E}} c(s,X_{s-},E[X_{s-}],e)\tilde{N}(\mathrm{d}s, \mathrm{d}e).
  \end{aligned}
\end{equation}
Here, $x_0\in R,\ b,\sigma,\gamma,c:\Omega\times [0,T] \times R\times R\times {\mathcal{E}}\rightarrow R$, and $G:[0,T]\rightarrow R$, and they satisfy the following
 assumptions.

\textbf{Assumption (A1)}.

$\bullet$ $b, \sigma, \gamma, c$ are $\mathscr{G}\otimes \mathscr{B}(R^2)\otimes \mathscr{B}({\mathcal{E}}) /\mathscr{B}(R)$ measurable, $G$ is $\mathscr{B}([0,T])$ measure continuous mapping.

$\bullet$ $b,\sigma,\gamma,c$ are uniformly Lipschitz continuous with respect to $(x,\tilde{x})$.

$\bullet$ $E(\int_0^T\int_{\mathcal{E}}|b(t,\omega,0,0,e)|\lambda(\mathrm{d}e)\mathrm{d}t)^2<\infty,\ E[\int_0^T(\int_{\mathcal{E}} |\sigma(t,\omega,0,0,e)|\lambda(\mathrm{d}e))^2\mathrm{d}t]<\infty$, \\
$E(\int_0^T\int_{\mathcal{E}} |\gamma(t,\omega,0,0,e)|N(\mathrm{d}t,\mathrm{d}e))^2<\infty,\ E[\int_0^T\int_{\mathcal{E}} |c(t,\omega,0,0,e)|^2 N(\mathrm{d}t,\mathrm{d}e)]<\infty $.

Moveover, we introduce a Banach space
\begin{equation*}
  S^2[0,T]\triangleq\bigg\{X:[0,T]\times \Omega\rightarrow R|X \text{ has c\`adl\`ag paths and adapted and } E\bigg[\sup_{0\leq t\leq T}|X_t|^2\bigg]<\infty\bigg\}
\end{equation*}
with norm $\|X\|^2=E[\sup_{0\leq t\leq T}|X_t|^2]$. Then, under Assumption (A1), we have the following existence-uniqueness and $L^2$ estimate in progressive structure about the MF-SDEP \eqref{SDE}.
\begin{lemma1}\label{SDEsolution-estimate}
  Under Assumption {\rm (A1)}, the mean-field SDEP \eqref{SDE} admits a unique solution in $S^2[0,T]$. Moveover, suppose that $X^i, i=1,2,$ are the solutions of the following equations
  \begin{equation}\label{xiii}
    \begin{aligned}
      X_t^i=x_0^i
      &+\int_0^t\int_{\mathcal{E}} b^i(s,X_s^i,E[X_s^i],e)\lambda(\mathrm{d}e)\mathrm{d}s+\int_0^t \int_{\mathcal{E}} \sigma^i(s,X_s^i,E[X_s^i],e)\lambda(\mathrm{d}e)\mathrm{d}B_s\\
      &+\int_0^t\int_{\mathcal{E}} \gamma^i(s,X_{s-}^i,E[X_{s-}^i],e)N(\mathrm{d}t,\mathrm{d}e) +\int_0^t\int_{\mathcal{E}} c^i(s,X_{s-}^i,E[X_{s-}^i],e)\tilde{N}(\mathrm{d}t,\mathrm{d}e),
    \end{aligned}
  \end{equation}
  respectively, and then we have
  \begin{equation}
    \begin{aligned}
      E\bigg[\sup_{0\leq t\leq T} |X_t^1-X_t^2|^2\bigg] &\leq C|x_0^1-x_0^2|^2+CE\bigg[\bigg(\int_0^T \bigg|\int_{\mathcal{E}} b^1(t,X_t^1,E[X_t^1],e)-b^2(t,X_t^1,E[X_t^1],e)\lambda(\mathrm{d}e)\bigg|\mathrm{d}t \bigg)^2\bigg]\\
      &\quad+CE\bigg[\int_0^T\bigg|\int_{\mathcal{E}}\sigma^1(t,X_t^1,E[X_t^1],e)-\sigma^2(t,X_t^1,E[X_t^1],e) \lambda(\mathrm{d}e)\bigg|^2 \mathrm{d}t\bigg]\\
      &\quad+CE\bigg[\bigg(\int_0^T\int_{\mathcal{E}}|\gamma^1(t,X_{t-}^1,E[X_{t-}^1],e)-\gamma^2(t,X_{t-}^1,E[X_{t-}^1],e)| N(\mathrm{d}t,\mathrm{d}e)\bigg)^2\bigg]\\
      &\quad+CE\bigg[\int_0^T\int_{\mathcal{E}}|c^1(t,X_{t-}^1,E[X_{t-}^1],e)-c^2(t,X_{t-}^1,E[X_{t-}^1],e)|^2 N(\mathrm{d}t,\mathrm{d}e)\bigg],
    \end{aligned}
  \end{equation}
  where $C$ is a positive real number related to $T$ and the Lipschitz constant.
\end{lemma1}
\begin{proof}
The proof we can be seen in Appendix A. 
\end{proof}

Next we give the existence and uniqueness of mean-field BSDE with jump and impulse in progressive structure (MF-BSDEP), which is given by the following form:
\begin{equation}\label{BSDE}
  Y_t=\xi+\int_t^T\int_{\mathcal{E}} f(s,\Lambda(s),E[\Lambda(s)],e)\lambda(\mathrm{d}e)\mathrm{d}s-\int_t^T Z_s\mathrm{d}B_s -\int_t^T\int_{\mathcal{E}} K_s \tilde{N}(\mathrm{d}s,\mathrm{d}e)-\int_t^TH_s\mathrm{d}\eta_s,
\end{equation}
where $\Lambda=(Y,Z,K)$ and $\xi$ is a $\mathscr{F}_T$-measurable random variable, $f:\Omega\times [0,T]\times R^3\times R^3\times {\mathcal{E}} \rightarrow R$, $H:[0,T]\rightarrow R$. Similarly, we first introduce two Banach spaces
\begin{equation*}
  M^2[0,T]\triangleq\bigg\{Z:[0,T]\times \Omega \rightarrow R]| Z \text{ is predictable and } E\left[\int_0^T|Z_s|^2 \mathrm{d}s \right]<\infty\bigg\}
\end{equation*}
with norm $\|Z\|^2=E[\int_0^T |Z_s|^2\mathrm{d}s]$, and
\begin{equation*}
  F^2[0,T]\triangleq\bigg\{K:[0,T]\times \Omega\times E\rightarrow R|K \text{ is $E$-predictable and } E\left[\int_0^T\int_{\mathcal{E}} |K_s|^2 \lambda(\mathrm{d}e)\mathrm{d}t \right]<\infty \bigg\}
\end{equation*}
with norm $\|K\|^2=E[\int_0^T\int_{\mathcal{E}} |K_t|^2 N(\mathrm{d}t,\mathrm{d}e)]$. In addition, for the backward system, we introduce the following additional assumptions:

\textbf{Assumption (A2)}.

$\bullet$ $f$ is $\mathscr{G}\otimes \mathscr{B}(R^3)\otimes \mathscr{B}(R^3)\otimes \mathscr{B}({\mathcal{E}})/\mathscr{B}(R)$ measurable, $H$ is $\mathscr{B}([0,T])$ measure continuous mapping.

$\bullet$ $f$ is uniformly Lipschitz with respect to $(y,z,k)$ and $(\tilde{y},\tilde{z},\tilde{k})$.

$\bullet$ $E(\int_0^T\int_{\mathcal{E}}|f(t,\omega,{\bf 0},e)|\lambda(\mathrm{d}e)\mathrm{d}t)^2<\infty$; $E[|\xi|^2]<\infty$, where ${\bf 0}$ is zero matrix with appropriate dimension.

Then, we can obtain the following results. Because the proof is similar to the Theorem 3.1 in \cite{Shen2013} and Theorem 2.8 in \cite{Song2021}, we omit the proof.
\begin{lemma1}\label{BSDEsolution}
  Under Assumption {\rm (A2)}, the mean-field BSDEP \eqref{BSDE} admits a unique solution in $S^2[0,T]\times M^2[0,T] \times F^2[0,T]$. Moreover, suppose that $Y^i,i=1,2$, are the solutions of the equations
  \begin{equation*}
    Y_t^i=\xi^i+\int_t^T\int_{\mathcal{E}}f^i(s,\Lambda^i(s),E[\Lambda^i(s)],e)\lambda(\mathrm{d}e)\mathrm{d}s-\int_t^T Z^i_s\mathrm{d}B_s-\int_t^T\int_{\mathcal{E}}K_s^i\tilde{N}(\mathrm{d}s,\mathrm{d}e),
  \end{equation*}
  respectively, then we have
  \begin{equation}
    \begin{aligned}
      &E\bigg[\sup_{0\leq t\leq T}|Y_t^1-Y_t^2|^2\bigg]+E\bigg[\int_0^T|Z_t^1-Z_t^2|^2\mathrm{d}t\bigg]+E\bigg[\int_0^T\int_{\mathcal{E}}|K_t^1 -K_t^2|^2\lambda(\mathrm{d}e)\mathrm{d}t\bigg]\\
      &\quad\leq CE\big[|\xi^1-\xi^2|^2\big]+CE\bigg[\bigg(\int_0^T\int_{\mathcal{E}}\big|f^1(t,\Lambda_t^1, E[\Lambda_t^1],e)- f^2(t,\Lambda_t^1,E[\Lambda_t^1],e) \big|\lambda(\mathrm{d}e)\mathrm{d}t\bigg)^2\bigg].
    \end{aligned}
  \end{equation}
\end{lemma1}

\section{Statement of the problem}

Let $\{T_n\}_{n\geq 1}$ be the jump time of $N([0,t]\times {\mathcal{E}}),\ T_n:=\inf\{t|N([0,t]\times {\mathcal{E}})\geq n\}$; then $\{T_n\}_{n\geq 1}$ is a sequence of stopping times that is strictly increasing. Let $U$ be a nonempty convex subset of $R$ and $\mathcal{K}$ be a nonempty convex subset of $\mathscr{K}$.
Consider the following progressive mean-field stochastic system with jumps and impulse
\begin{equation}\label{state}
  \left\{
    \begin{aligned}
      \mathrm{d}x_t&= \int_{\mathcal{E}}b(t,x_t,E[x_t],u_{(t,e)},e)\lambda(\mathrm{d}e)\mathrm{d}t +\int_{\mathcal{E}}\sigma(t,x_t,E[x_t],u_{(t,e)},e) \lambda(\mathrm{d}e)\mathrm{d}B_t+G_t\mathrm{d}\eta_t\\
      &\qquad+\int_{\mathcal{E}} \gamma(t,x_{t-},E[x_{t-}],u_{(t,e)},e)N(\mathrm{d}t,\mathrm{d}e)
      +\int_{\mathcal{E}} c(t,x_{t-},E[x_{t-}],u_{(t,e)},e)\tilde{N}(\mathrm{d}t,\mathrm{d}e),\\
      \mathrm{d}y_t&=-\int_{\mathcal{E}}g(t,x_t,E[x_t],y_t,E[y_t],z_t,E[z_t],u_{(t,e)},e)\lambda(\mathrm{d}e) \mathrm{d}t+H_t\mathrm{d}\eta_t+z_t\mathrm{d}B_t +\int_{\mathcal{E}}k_t\tilde{N}(\mathrm{d}t,\mathrm{d}e),\\
      x_0&=x_0,\qquad y_T=h(x_T,E[x_T]),
    \end{aligned}
  \right.
\end{equation}
where $x_0$ is a constant and $b,\sigma, \gamma, c: \Omega\times [0,T]\times R^2\times U\times {\mathcal{E}}\rightarrow R$; $g: \Omega\times[0,T]\times R^6\times U\times {\mathcal{E}}\rightarrow R$; $G,H:[0,T]\rightarrow R$, $h:\Omega\times R^2\rightarrow R$. We define the following control set
\begin{equation*}
  \begin{aligned}
    \mathcal{U}_{ad}&=\bigg\{u| u \text{ is } \text{progressive,} \text{ taking values in } U, E\bigg[\int_0^T\bigg(\int_{\mathcal{E}}|u_{(t,e)}|\lambda(\mathrm{d}e)\bigg)^2 \mathrm{d}t\bigg]<\infty\\ &\qquad\qquad\qquad\qquad\qquad\qquad \qquad\qquad\qquad E\bigg[\bigg(\int_0^T\int_{\mathcal{E}}|u_{(t,e)}|N(\mathrm{d}t,\mathrm{d}e)\bigg)^2\bigg]<\infty\bigg\},\\ 
    \mathcal{K}_{ad}&=\bigg\{\eta|\eta\text{ is adapted,} \text{ taking values in } \mathcal{K} \text{ such that } E \bigg[\bigg( \sum_{i\geq 1}|\eta_{i}|\bigg)^2\bigg]<\infty\bigg\}.
  \end{aligned}
\end{equation*}
We call $\mathcal{A}:=\mathcal{U}_{ad}\times \mathcal{K}_{ad}$ the admissible control.

\begin{remark1}\label{remark331}
  {\rm i)} In the existing work, the control variables do not depend on the variable ``e". As far as we know, the variable ``e" are usually used to describe the amplitude of jump process. For example, the system is affected by L$\acute{e}$vy process. $m$ denotes the number of times that the jump amplitude is less than $a$ before time $t$, which can be expressed by a Poisson random measure, that is $m=N([0,t]\times[0,a])$. Our model is mainly concerned with the control of jump time, so it is reasonable that the control is related to the amplitude of jump process.\\
  {\rm ii)} Let $(T_n,U_n)$ be an atom of Poisson random measure. For any $u\in\mathcal{U}_{ad}$, we have
  \begin{equation*}
    \begin{aligned}
      &E\bigg[\int_0^T\int_{\mathcal{E}}|u|^2N(\mathrm{d}t,\mathrm{d}e)\bigg]=E\bigg[\sum_{n=1}^{\infty}|u(\omega, T_n(\omega),U_n(\omega))|^2\bigg]\\
      &\leq E\bigg[\bigg(\sum_{n=1}^{\infty}|u(\omega, T_n(\omega),U_n(\omega))|\bigg)^2\bigg] =E\bigg[\bigg(\int_0^T\int_{\mathcal{E}}|u|N(\mathrm{d}t,\mathrm{d}e)\bigg)^2\bigg]<\infty.
    \end{aligned}
  \end{equation*}
  {\rm iii)} In order to simplify the symbols and assumptions, the state system in this paper is a decoupled forward-backward stochastic control system. In actually, our results can be extended to fully coupled forward-backward systems (i.e. the coefficients $b,\sigma,\gamma,c$ depend on $y_{\cdot}$ and $z_{\cdot}$), which does not present essential difficulties.
\end{remark1}
In order to ensure the well-posedness of the controlled system, we introduce the following assumptions.

\textbf{Assumption (H1)}

$\bullet$ $b,\sigma,\gamma,c$ are $\mathscr{G}\otimes \mathscr{B}(R^2)\otimes \mathcal{U}_{ad}\otimes \mathscr{B}({\mathcal{E}})/\mathscr{B}(R)$ measurable; $g$ is $\mathscr{G}\otimes \mathscr{B}(R^6)\otimes \mathcal{U}_{ad}\otimes \mathscr{B}({\mathcal{E}})/\mathscr{B}(R)$ measurable; $h$ is $\mathscr{F}_T\otimes \mathscr{B}(R^2)/\mathscr{B}(R)$ measurable; $G,H$ are $\mathscr{B}([0,T])$ measure continuous mapping.

$\bullet$ $b,\sigma,\gamma,c$ are continuously differentiable with respect to  $(x,\tilde{x},u)$, respectively, $g$ is continuously differentiable with respect to  $(\Lambda,\tilde{\Lambda},u)$, and their partial derivatives are uniformly bounded.

$\bullet$ $E\int_0^T [(\int_{\mathcal{E}}|b(t,\omega,{\bf 0},e)|\lambda(\mathrm{d}e))^2+(\int_{\mathcal{E}}|\sigma (t,\omega,{\bf 0},e)|\lambda(\mathrm{d}e))^2+(\int_{\mathcal{E}}|g(t,\omega,{\bf 0},e)|\lambda(\mathrm{d}e))^2]\mathrm{d}t <\infty$,\\
$E[\int_0^T \int_{\mathcal{E}}|\gamma(t,\omega,{\bf 0},e)|N(\mathrm{d}t,
\mathrm{d}e)]^2 <\infty$, $E\int_0^T\int_{\mathcal{E}} |c(t,\omega,{\bf 0},e)|^2N(\mathrm{d}t,\mathrm{d}e)<\infty$.

Under Assumption (H1), we know that there exits a unique solution of \eqref{state} for any admissible control from Lemma \ref{SDEsolution-estimate} and Lemma \ref{BSDEsolution}.  Furthermore, we here consider the following cost functional
\begin{equation}\label{cost}
  \begin{aligned}
    J({u,\eta})&=E\bigg[\int_0^T\int_{\mathcal{E}}l(t,\Lambda_t,E[\Lambda_t],u_{(t,e)},e)\lambda(\mathrm{d}e) \mathrm{d}t +\varphi(x_T,E[x_T])+\phi(y_0)\\
    &\qquad\qquad\qquad+\int_0^T\int_{\mathcal{E}} f(t,\Lambda_{t-},E[\Lambda_{t-}],u_{(t,e)},e)N(\mathrm{d}t,\mathrm{d}e)+\sum_{i\geq 1}\psi(\tau_i,\eta_i)\bigg],
  \end{aligned}
\end{equation}
where $l,f:\Omega\times [0,T]\times R^6\times U\times {\mathcal{E}}\rightarrow R$, $\varphi:\Omega\times R^2\rightarrow R$, $\phi:\Omega\times R\rightarrow R$, $\psi:[0,T]\times R\rightarrow R$. Here, for the cost functional, we need the following assumptions.

\textbf{Assumption (H2)}.

$\bullet$ $l,f$ are $\mathscr{G}\otimes \mathscr{B}(R^6)\otimes \mathscr{B}(U)\otimes \mathscr{B}({\mathcal{E}})/ \mathscr{B}(R)$ measurable, $\varphi$ is $\mathscr{F}_T\otimes \mathscr{B}(R^2)/\mathscr{B}(R)$ measurable, $\phi$ is $\mathscr{F}_0\otimes \mathscr{B}(R)/\mathscr{B}(R)$ measurable, $\psi$ is $\mathscr{F}_T\otimes \mathscr{B}(R)/\mathscr{B}(R)$ measurable.

$\bullet$ $l,f,\varphi$ are continuous differentiable with respect to $(\Lambda,\tilde{\Lambda},u)$ and there exist a constant $\bar{C}$ such that $|(l,f)
|\leq \bar{C}(1+|\Lambda|^2+|\tilde{\Lambda}|^2+|u|^2)$, $|(l,f)_{(\Lambda,\tilde{\Lambda},u)}
|\leq \bar{C}(1+\Lambda+|\tilde{\Lambda}|+|u|)$, $|\varphi| \leq \bar{C}(1+|x|^2+|\tilde{x}|^2)$, $|\varphi_{(x,\tilde{x},u)}|\leq \bar{C}(1+|x|+|\tilde{x}|)$, $|\phi|\leq \bar{C}(1+|y|^2)$, $|\phi_y|\leq \bar{C}(1+|y|)$, $|\psi|\leq \bar{C}(1+|\eta|^2)$, $|\psi_{\eta}|\leq \bar{C}(1+|\eta|)$.

For convenience, we call the coefficients satisfy Assumption (H) if and only if they satisfy Assumptions (H1) and (H2).
Then, the optimal control problem studied in this paper can be described as:

{\bf Problem (OCP)}
  Finding an admissible control $(\hat u,\hat\eta)\in \mathcal{A}$ such that
  \begin{equation*}
  J(\hat u,\hat \eta)=\inf_{(u,\eta)\in\mathcal{A}}J(u,\eta).
\end{equation*}

\begin{remark1}
  Let Assumption {\rm (H)} hold in predictable structure, i.e. $\gamma $ is $\mathscr{P}\otimes \mathscr{B}(R^2)\otimes \mathscr{B}(U)\otimes\mathscr{B}({\mathcal{E}})/\mathscr{B}(R)$ measurable, then we have the following result for any $X\in S^2[0,T]$ and any $(u,\eta)\in \mathcal{A}$
  \begin{equation*}
    |\gamma(t,x_{t-},E[x_{t-}],u_{(t,e)},e)|\leq \bar{C}(|\gamma(t,0,0,0,e)|+|x_{t-}|+|E[x_{t-}]| +|u_{(t,e)}|).
  \end{equation*}
  According to Proposition \ref{2.2}, we have
  \begin{equation*}
    \begin{aligned}
      &\int_0^t\int_{\mathcal{E}}\gamma(s,x_{s-},E[x_{s-}],u_{(s,e)},e)N(\mathrm{d}s,\mathrm{d}e)\\
      &=\int_0^t\int_{\mathcal{E}}\gamma (s,x_{s-},E[x_{s-}],u_{(s,e)},e)\tilde{N}(\mathrm{d}s,\mathrm{d}e)+\int_0^t\int_{\mathcal{E}} \check{\mathbb{E}}[\gamma(s,x_{s-},E[x_{s-}],u_{(s,e)},e)]\lambda(\mathrm{d}e)\mathrm{d}s\\
      &=\int_0^t\int_{\mathcal{E}}\gamma (s,x_{s-},E[x_{s-}],u_{(s,e)},e)\tilde{N}(\mathrm{d}s,\mathrm{d}e)+\int_0^t\int_{\mathcal{E}} \gamma(s,x_{s},E[x_{s}],u_{(s,e)},e)\lambda(\mathrm{d}e)\mathrm{d}s.
    \end{aligned}
  \end{equation*}
  Then the first equation of system \eqref{state} can be rewritten as
  \begin{equation*}
    \left\{
      \begin{aligned}
        \mathrm{d}x_t&=\int_{\mathcal{E}}(b+\gamma)(t,x_t,E[x_t],u_{(t,e)},e) \lambda(\mathrm{d}e)\mathrm{d}t +\int_{\mathcal{E}}\sigma(t,x_t,E[x_t],u_{(t,e)},e)\lambda(\mathrm{d}e)\mathrm{d}B_t\\ &\quad\quad+G_t\mathrm{d}\eta_t+\int_{\mathcal{E}}(\gamma+c)(t,x_{t-},E[x_{t-}],u_{(t,e)},e) \tilde{N}(\mathrm{d}t,\mathrm{d}e),\\
        x_0&=x_0.
      \end{aligned}
    \right.
  \end{equation*}
  This shows that $\gamma$ can be omitted. If $f$ is $\mathscr{P}\otimes \mathscr{B}(R)\otimes\mathscr{B}(R) \otimes\mathscr{B}(U)\otimes\mathscr{B}({\mathcal{E}})/\mathscr{B}(R)$ measurable, we have
  \begin{equation*}
    \begin{aligned}
      J({u,\eta})=E\bigg[\int_0^T\int_{\mathcal{E}}(l+f)(t,\Lambda_t,E[\Lambda_t],&u_{(t,e)},e) \lambda(\mathrm{d}e)\mathrm{d}t
      +\varphi(x_T,E[x_T])+\phi(y_0)+\sum_{i\geq 1}\psi(\tau_i,\eta_i) \bigg].
    \end{aligned}
  \end{equation*}
  We also know that $f$ can be omitted in the cost function.
\end{remark1}

\section{Variation}

Let $(\hat{u},\hat{\eta}=\sum_{i\geq 1}\hat{\eta}_i\mathbbm{1}_{[\tau_i,T]})$ be an optimal control for the control problem {\bf (OCP)}, and the corresponding trajectories of dynamics are denoted by $(\hat{x},\hat{y},\hat{z},\hat{k})$. Let $(v,\xi=\sum_{i\geq 1}\xi_i\mathbbm{1}_{[\tau_i,T]}) \in \mathcal{A}$ be another control process. For the reason that the control domains $\mathcal{U}$ and $\mathcal{K}$ are convex, we have for any $0\leq \epsilon_1 \leq 1$, $\ u^{\epsilon_1}=(1-\epsilon_1)\hat{u}+\epsilon_1 v\in \mathcal{U}$, for any $0\leq \epsilon_2 \leq 1$, $\eta^{\epsilon_2}=(1-\epsilon_2)\hat{\eta}+\epsilon_2 \xi\in\mathcal{K}$, the corresponding trajectory is denoted by $(x^{\epsilon},y^{\epsilon},z^{\epsilon},k^{\epsilon})$.

For simplification, we introduce the following notations,
\begin{equation*}
  \left\{
    \begin{aligned}
      &\alpha(t)\triangleq \alpha(t,\hat{x}_t,E[\hat{x}_t],\hat{u}_{(t,e)},e),\beta(t)\triangleq \beta(t,\hat{x}_{t-},E[\hat{x}_{t-}],\hat{u}_{(t,e)},e),\\
      &\alpha_{\iota}(t)\triangleq \alpha_{\iota}(t,\hat{x}_t,E[\hat{x}_t],\hat{u}_{(t,e)},e),\beta_{\iota}(t)\triangleq \beta_{\iota}(t,\hat{x}_{t-},E[\hat{x}_{t-}],\hat{u}_{(t,e)},e),\\
      &\alpha_\iota^{\theta}(t)\triangleq\alpha_\iota^{\theta}(t,\theta x_t^{\epsilon}+(1-\theta)\hat{x}_t,E[\theta x_t^{\epsilon}+(1-\theta)\hat{x}_t],\theta u_{(t,e)}^{\epsilon_1}+(1-\theta)\hat{u}_{(t,e)},e),\\
      &\beta_\iota^{\theta}(t)\triangleq\beta_\iota^{\theta}(t,\theta x_{t-}^{\epsilon}+(1-\theta) \hat{x}_{t-},E[\theta x_{t-}^{\epsilon}+(1-\theta)\hat{x}_{t-}],\theta u_{(t,e)}^{\epsilon_1}+(1-\theta)\hat{u}_{(t,e)},e),\\
      &\delta \alpha_\iota^{\theta}(t)\triangleq\alpha_\iota^{\theta}(t)-\alpha_\iota(t),\delta \beta_\iota^{\theta}(t)\triangleq\beta_\iota^{\theta}(t)-\beta_\iota(t),\\
      &\zeta_{\varpi}(t)\triangleq \zeta_{\varpi}(t,\hat{\Lambda}_t,E[\hat{\Lambda}_t], \hat{u}_{(t,e)},e),f_{\varpi}(t)\triangleq f_{\varpi}(t,\hat{\Lambda}_{t-},E[\hat{\Lambda}_{t-}],\hat{u}_{(t,e)},e),
    \end{aligned}
  \right.
\end{equation*}
where $\alpha = b,\sigma,\varphi$, $\beta=\gamma,c$, $\zeta=g,l,f$, $\iota=x,\tilde{x},u$, $\varpi=\Lambda,\tilde{\Lambda},u$.

Moreover, we introduce the following variational equation
\begin{equation}\label{varequ}
  \left\{
  \begin{aligned}
    \mathrm{d}{x}^1_t &= \int_{\mathcal{E}}\Big\{b_x(t){x}^1_t+b_{\tilde{x}}(t)E[{x}^1_t]+b_u(t)\hat{v}_t\Big\}\lambda(\mathrm{d}e)\mathrm{d}t 
    + \int_{\mathcal{E}}\Big\{\sigma_x(t){x}^1_t+\sigma_{\tilde{x}}(t)E[{x}^1_t] \\
    &\qquad+\sigma_u(t)\hat{v}_t\Big\}\lambda(\mathrm{d}e)\mathrm{d}B_t+ \int_{\mathcal{E}}\Big\{\gamma_x(t){x}^1_{t-}+\gamma_{\tilde{x}}(t) E[{x}^1_{t-}]+\gamma_u(t)\hat{v}_t\Big\}N(\mathrm{d}t,\mathrm{d}e)\\
    &\qquad+ \int_{\mathcal{E}} \Big\{c_x(t){x}^1_{t-}+c_{\tilde{x}}(t) E[{x}^1_{t-}]+c_u(t)\hat{v}_t\Big\}\tilde{N}(\mathrm{d}t,\mathrm{d}e)+G_t\mathrm{d}\hat{\xi}_t,\\
    \mathrm{d}y_t^1&=-\int_{\mathcal{E}}\Big\{g_x(t)x_t^1+g_{\tilde{x}}(t)E[x_t^1]+g_y(t)y_t^1+g_{\tilde{y}}(t)E[y_t^1] +g_z(t)z_t^1 +g_{\tilde{z}}(t)E[z_t^1]\\
    &\qquad+g_u(t)\hat{v}_t  \Big\}\lambda(\mathrm{d}e)\mathrm{d}t+H_t\mathrm{d}\hat{\xi}_t +z_t^1\mathrm{d}B_t +\int_{\mathcal{E}}k_t^1\tilde{N}(\mathrm{d}t,\mathrm{d}e),\\
    x_0^1=&0,\qquad y_T^1=h_x(T) x_T^1+h_{\tilde x}E[x_T^1],
  \end{aligned}
  \right.
\end{equation}
where $\hat{v}_t=\epsilon_1(v_{(t,e)}-\hat{u}_{(t,e)})$ and $\hat{\xi}=\epsilon_2(\xi-\hat{\eta})$. Since equation \eqref{varequ} is a decoupled forward-backward stochastic differential equation with jump and impulse (FBSDEP) with bounded coefficients, we obtain that mean-field FBSDEP \eqref{varequ} has unique solution.

\begin{remark1}
  It is worth noting that we do not require the coefficients of disturbances to be the same i.e. $\epsilon_1=\epsilon_2$. This is different from the existing literature. But our results can degenerate into the case of $\epsilon_1=\epsilon_2=\epsilon$.
\end{remark1}

\begin{lemma1}
Let Assumption {\rm(H)} hold, then we have the following estimates.
  \begin{equation*}
    \begin{aligned}
      &E\bigg[\sup_{0\leq t\leq T}|x_t^1|^2\bigg]\leq C(\epsilon_1^2+\epsilon_2^2),\qquad
      &&E\bigg[\sup_{0\leq t\leq T}|y_t^1|^2\bigg]\leq C(\epsilon_1^2+\epsilon_2^2),\\
      &E\bigg[\int_0^T|z_t^1|^2\mathrm{d}t\bigg]\leq C(\epsilon_1^2+\epsilon_2^2),\qquad
      &&E\bigg[\int_0^T\int_{\mathcal{E}}|k_t^1|^2\lambda(\mathrm{d}e)\mathrm{d}t\bigg]\leq C(\epsilon_1^2+\epsilon_2^2),
    \end{aligned}
  \end{equation*}
  where $C$ is a constant independent of $\epsilon_1$ and $\epsilon_2$.
\end{lemma1}
\begin{proof}
  By the $L^2$ estimate in Lemma \ref{SDEsolution-estimate}, for $x_t^1$ we have
  \begin{equation*}
    \begin{aligned}
      E\bigg[\sup_{0\leq t\leq T}|x_t^1|^2\bigg]&\leq CE\bigg[\int_0^T\bigg|\int_{\mathcal{E}}b_u(t) \hat{v}_t\lambda (\mathrm{d}e)\bigg|^2\mathrm{d}t+\int_0^T\bigg| \int_{\mathcal{E}}\sigma_u(t)\hat{v}_t \lambda(\mathrm{d}e)\bigg|^2\mathrm{d}t+\bigg|\int_0^TG_t\mathrm{d} \hat{\xi}_t \bigg|^2\bigg]\\
      &\qquad+E\bigg[\bigg(\int_0^T\int_{\mathcal{E}}|\gamma_u(t)\hat{v}_t|N(\mathrm{d}t,\mathrm{d}e)\bigg)^2\bigg]
      +E\bigg[\int_0^T\int_{\mathcal{E}}|c_u(t)\hat{v}_t|^2N(\mathrm{d}t,\mathrm{d}e)\bigg]\\
      &\leq CE\bigg[\epsilon_1^2\int_0^T\bigg|\int_{\mathcal{E}}(\delta v_{(t,e)})\lambda(\mathrm{d}e) \bigg|^2\mathrm{d}t+\epsilon_1^2\bigg(\int_0^T\int_{\mathcal{E}}|\delta v_{(t,e)}|N(\mathrm{d}t,
      \mathrm{d}e)\bigg)^2\\
      &\qquad +\epsilon_1^2\int_0^T\int_{\mathcal{E}}|\delta v_{(t,e)}|^2N(\mathrm{d}t,
      \mathrm{d}e)+\epsilon_2^2\bigg(\sum_{i\geq 1}|\xi_i-\hat{\eta}_i|\bigg)^2\bigg]\\
      &\leq C(\epsilon_1^2+\epsilon_2^2),
    \end{aligned}
  \end{equation*}
  where $\delta v_{(t,e)}=v_{(t,e)}-\hat u_{(t,e)}$. The last inequality is due to Remark \ref{remark331}. By the classical $L^2$ estimate of BSDEP, we obtain
  \begin{equation*}
    \begin{aligned}
      &E\bigg[\sup_{0\leq t\leq T}|y_t^1|^2\bigg]+E\bigg[\int_0^T|z_t^1|^2\mathrm{d}t\bigg] +E\bigg[\int_0^T\int_{\mathcal{E}}|k_t^1|^2\lambda(\mathrm{d}e)\mathrm{d}t\bigg]\\
      &\leq CE[|x_T^1|^2]+CE\bigg[\int_0^T\bigg|\int_{\mathcal{E}}g_u(t)\hat{v}_t\lambda(\mathrm{d}e)\bigg|^2 \mathrm{d}t\bigg]+CE\bigg[\bigg(\int_0^TH_t\mathrm{d}\hat{\xi}_t\bigg)^2\bigg]\\
      &\leq C(\epsilon_1^2+\epsilon_2^2).
    \end{aligned}
  \end{equation*}
  Then the proof is complete.
\end{proof}
Now we first present an important proposition, which is of great significance for subsequent results. 
\begin{proposition1}\label{prop433}
  Let Assumption {\rm(H)} hold, we have
  \begin{equation*}
    \lim_{(\epsilon_1,\epsilon_2)\rightarrow 0}\delta \alpha_\iota^{\theta}(t)=0,\qquad \text{and} \qquad \lim_{(\epsilon_1,\epsilon_2)\rightarrow 0}\delta \beta_\iota^{\theta}(t)=0,
  \end{equation*}
  with $\alpha=b,\sigma$ and $\beta=\gamma,c$. The limits are convergent in $\mu\times Leb+\nu\times Leb$, here $\nu=\mathbb{P}\times Leb\times \lambda$ is a measure on $\Omega\times [0,T]\times {\mathcal{E}}$.
\end{proposition1}
\begin{proof}
 Let $\tilde \epsilon=(\epsilon_1,\epsilon_2)$, for any sequence $\{\tilde \epsilon_n\}$ converges to $0$ as $n\rightarrow +\infty$, noticing
 \begin{equation*}
     \lim_{\epsilon_n\rightarrow 0}E\bigg[\sup_{0\leq t\leq T}\big|x_t^{\tilde \epsilon_n}-\hat x_t\big|^2\bigg]=0,
 \end{equation*}
 then we can know that there exists a subsequence $\{\tilde\epsilon_{n_k}\}$ of $\{\tilde\epsilon_n\}$ such that $x_t^{\tilde\epsilon_{n_k}}\rightarrow x_t$ a.s. as $k\rightarrow +\infty$ for any $t\in [0,T]$. For any sequence $\{\tilde\epsilon_n\}$ satisfies $\tilde\epsilon_n\rightarrow 0$ as $n\rightarrow +\infty$, we define
 \begin{equation*}
     \alpha_{\iota}^{n,\theta}(t)\triangleq \alpha_{\iota}(t,\theta x_t^{\tilde\epsilon_n}+(1-\theta)\hat{x}_t,E[\theta x_t^{\tilde\epsilon_n}+(1-\theta)\hat{x}_t],\theta u_{(t,e)}^{\tilde \epsilon_n}+(1-\theta)\hat{u}_{(t,e)},e).
 \end{equation*}
 For any subsequence $\{\tilde\epsilon_{n_k}\}$ of $\{\tilde\epsilon_n\}$, we have $\tilde\epsilon_{n_k}\rightarrow 0$ as $k\rightarrow +\infty$. Then we apply the result at the beginning of this proof, we obtain that there exists a subsequence $\{\tilde\epsilon_{n_{k_l}}\}$ such that $x_t^{\tilde\epsilon_{n_{k_l}}}\rightarrow \hat x_t$ a.s. as $\tilde\epsilon_{n_{k_l}}\rightarrow 0$. Let $A$ denote the convergent set, then $\mathbb P(A)=1$. Since $\alpha_{\iota}$ is continuous with respect to $(x,\tilde x,u)$, for each $t\in [0,T]$, $e\in \mathcal E$, $\theta \in [0,1]$, $\omega\in A$, we have
 \begin{equation*}
     \lim_{\tilde\epsilon_{n_{k_l}}\rightarrow 0}\alpha_{\iota}(t,\theta x_t^{\tilde\epsilon_{n_{k_l}}}+(1-\theta)\hat{x}_t,E[\theta x_t^{\tilde\epsilon_{n_{k_l}}}+(1-\theta)\hat{x}_t],\theta u_{(t,e)}^{\tilde \epsilon_{n_{k_l}}}+(1-\theta)\hat{u}_{(t,e)},e)=\alpha_{\iota}(t).
 \end{equation*}
 Thus we have $\alpha_{\iota}^{n_{k_l},\theta}(t)-\alpha_{\iota}(t)$ point converge to $0$ on $A\times [0,T]\times [0,1]$. Due to
 \begin{equation*}
     (\mu\times Leb+\nu \times Leb)(A\times [0,T]\times \mathcal E\times [0,1])=(\mu\times Leb+\nu \times Leb)(\Omega \times [0,T]\times \mathcal E\times [0,1])
 \end{equation*}
 We obtain that for any subsequence of $\alpha_{\iota}^{n,\theta}-\alpha_{\iota}(t)$, there exists a subsequence of this subsequence converges to $0$ a.s. in $\mu\times Leb+\nu\times Leb$. Therefore $\alpha_{\iota}^{n,\theta}-\alpha_{\iota}(t)$ converges to $0$ in $\mu\times Leb+\nu\times Leb$.

 Secondly, since
 \begin{equation*}
     \lim_{\tilde\epsilon_n}E\bigg[\sup_{0\leq t\leq T}\big|x_{t-}^{\tilde\epsilon_n}-x_{t-}\big|^2\bigg]\leq \lim_{\tilde\epsilon_n}E\bigg[\sup_{0\leq t\leq T}\big|x_{t}^{\tilde\epsilon_n}-x_{t}\big|^2\bigg]=0,
 \end{equation*}
 and can find a subsequence $\{\tilde\epsilon_{n_k}\}$ of $\{\tilde\epsilon_n\}$ such that $x_{t-}^{\tilde\epsilon_{n_k}}$ a.s. converges to $x_{t-}$, then we define
 \begin{equation*}
     \beta_{\iota}^{n,\theta}(t)\triangleq \beta_{\iota}(t,\theta x_{t-}^{\tilde\epsilon_n}+(1-\theta)\hat{x}_{t-},E[\theta x_{t-}^{\tilde\epsilon_n}+(1-\theta)\hat{x}_{t-}],\theta u_{(t,e)}^{\tilde \epsilon_n}+(1-\theta)\hat{u}_{(t,e)},e).
 \end{equation*}
 We can obtain the conclusion by the same argument in the first step.
\end{proof}
\begin{lemma1}\label{lemma422}
Let Assumption {\rm(H)} hold, then we have the following estimates.
\begin{equation*}
  \begin{aligned}
    &\lim_{(\epsilon_1,\epsilon_2)\rightarrow 0}\frac{1}{\epsilon_1^2+\epsilon_2^2} E\bigg[\sup_{0\leq t\leq T}|x_t^{\epsilon}-\hat{x}_t-x_t^1|^2\bigg]=0,\\
    &\lim_{(\epsilon_1,\epsilon_2)\rightarrow 0}\frac{1}{\epsilon_1^2+\epsilon_2^2} E\bigg[\sup_{0\leq t\leq T}|y_t^{\epsilon}-\hat{y}_t-y_t^1|^2\bigg]=0,\\
    &\lim_{(\epsilon_1,\epsilon_2)\rightarrow 0}\frac{1}{\epsilon_1^2+\epsilon_2^2} E\bigg[\int_0^T|z_t^{\epsilon}-\hat{z}_t-z_t^1|^2\mathrm{d}t\bigg]=0,\\
    &\lim_{(\epsilon_1,\epsilon_2)\rightarrow 0}\frac{1}{\epsilon_1^2+\epsilon_2^2} E\bigg[\int_0^T\int_{\mathcal{E}}|k_t^{\epsilon}-\hat{k}_t-k_t^1|^2\lambda(\mathrm{d}e)\mathrm{d}t\bigg]=0.
  \end{aligned}
\end{equation*}
\end{lemma1}
\begin{proof}
  Firstly, we give the equation that $\delta x_t^{\epsilon}=x_t^\epsilon-\hat{x}_t$ satisfies
  \begin{equation*}
    \begin{aligned}
      \delta x_t^{\epsilon}&=\int_0^t\int_{\mathcal{E}}\bigg\{\int_0^1b_x^{\theta}(s)\mathrm{d}\theta\cdot \delta x_t^{\epsilon}+\int_0^1b_{\tilde x}^{\theta}(s)\mathrm{d}\theta\cdot E[\delta x_t^{\epsilon}]+\int_0^1b_u^{\theta}(s)\mathrm{d}\theta\cdot \hat{v}_s \bigg\}\lambda(\mathrm{d}e)\mathrm{d}s\\
      &\quad+\int_0^t\int_{\mathcal{E}}\bigg\{\int_0^1\sigma_x^{\theta}(s)\mathrm{d}\theta\cdot \delta x_t^{\epsilon}+\int_0^1\sigma_{\tilde x}^{\theta}(s)\mathrm{d}\theta\cdot E[\delta x_t^{\epsilon}]+\int_0^1\sigma_u^{\theta}(s)\mathrm{d}\theta\cdot \hat{v}_s \bigg\}\lambda(\mathrm{d}e)\mathrm{d}B_s\\
      &\quad+\int_0^t\int_{\mathcal{E}}\bigg\{\int_0^1\gamma_x^{\theta}(s)\mathrm{d}\theta\cdot \delta x_{t-}^{\epsilon}+\int_0^1\gamma_{\tilde x}^{\theta}(s)\mathrm{d}\theta\cdot E[\delta x_{t-}^{\epsilon}]+\int_0^1\gamma_u^{\theta}(s)\mathrm{d}\theta\cdot \hat{v}_s \bigg\}N(\mathrm{d}t,\mathrm{d}e)\\
      &\quad+\int_0^t\int_{\mathcal{E}}\bigg\{\int_0^1c_x^{\theta}(s)\mathrm{d}\theta\cdot \delta x_{t-}^{\epsilon}+\int_0^1c_{\tilde x}^{\theta}(s)\mathrm{d}\theta\cdot E[\delta x_{t-}^{\epsilon}]+\int_0^1c_u^{\theta}(s)\mathrm{d}\theta\cdot \hat{v}_s \bigg\}\tilde{N}(\mathrm{d}t,\mathrm{d}e).
    \end{aligned}
  \end{equation*}
  By Lemma \ref{SDEsolution-estimate}, we obtain
  \begin{equation*}
    \begin{aligned}
      &E\bigg[\sup_{0\leq t\leq T}|\delta x_t^{\epsilon}-x_t^1|^2\bigg]\\
      &\leq CE\bigg[\int_0^t\int_{\mathcal{E}}\int_0^1\Big\{|\delta b_x^{\theta}(s)|^2|x_t^1|^2+|\delta b_{\tilde x}^{\theta}(s)|^2E[|x_t^1|^2]+|\delta b_u^{\theta}(s)|^2|\hat{v}_s|^2 \Big\}\mathrm{d}\theta\lambda(\mathrm{d}e)\mathrm{d}s\bigg]\\
      &\quad +CE\bigg[\int_0^t\int_{\mathcal{E}}\int_0^1\Big\{|\delta\sigma_x^{\theta}(s)|^2 |x_t^1|^2 +|\delta\sigma_{\tilde x}^{\theta}(s)|^2E[| x_t^1|^2] +|\delta\sigma_u^{\theta}(s)|^2 |\hat{v}_s|^2 \Big\}\mathrm{d}\theta\lambda(\mathrm{d}e)\mathrm{d}B_s\bigg]\\
      &\quad +CE\bigg[\int_0^t\int_{\mathcal{E}}\int_0^1\Big\{|\delta\gamma_x^{\theta}(s)|  |x_{t-}^1| +|\delta\gamma_{\tilde x}^{\theta}(s)| E[| x_{t-}^1|] +|\delta\gamma_u^{\theta}(s)|| \hat{v}_s| \Big\}\mathrm{d}\theta N(\mathrm{d}t,\mathrm{d}e)\bigg]^2\\
      &\quad +CE\bigg[\int_0^t\int_{\mathcal{E}}\int_0^1\Big\{|\delta c_x^{\theta}(s)|^2  |x_{t-}^1|^2 +|\delta c_{\tilde x}^{\theta}(s)|^2 E[| x_{t-}^1|^2]+|\delta c_u^{\theta}(s)|^2|\hat{v}_s |^2 \Big\}\mathrm{d}\theta\tilde{N}(\mathrm{d}t,\mathrm{d}e)\bigg]\\
      &\leq C(\epsilon_1^2+\epsilon_2^2)E\bigg[\int_0^t\int_{\mathcal{E}}\int_0^1\Big\{|\delta b_x^{\theta}(s)|^2+|\delta b_{\tilde x}^{\theta}(s)|^2+|\delta b_u^{\theta}(s)|^2|\delta v_{(s,e)} |^2 \Big\}\mathrm{d}\theta\lambda(\mathrm{d}e)\mathrm{d}s\bigg]\\
      &\quad +C(\epsilon_1^2+\epsilon_2^2)E\bigg[\int_0^t\int_{\mathcal{E}}\int_0^1\Big\{|\delta \sigma_x^{\theta}(s)|^2+|\delta\sigma_{\tilde x}^{\theta}(s)|^2 +|\delta\sigma_u^{\theta}(s)|^2 |\delta v_{(s,e)}|^2 \Big\}\mathrm{d}\theta\lambda(\mathrm{d}e)\mathrm{d}B_s\bigg]\\
      &\quad +C(\epsilon_1^2+\epsilon_2^2)E\bigg[\int_0^t\int_{\mathcal{E}}\int_0^1\Big\{|\delta\gamma_x^{\theta}(s)|  +|\delta\gamma_{\tilde x}^{\theta}(s)| +|\delta\gamma_u^{\theta}(s)||\delta{v}_{(s,e)}| \Big\}\mathrm{d}\theta N(\mathrm{d}t,\mathrm{d}e)\bigg]^2\\
      &\quad +C(\epsilon_1^2+\epsilon_2^2)E\bigg[\int_0^t\int_{\mathcal{E}}\int_0^1\Big\{|\delta c_x^{\theta}(s)|^2  +|\delta c_{\tilde x}^{\theta}(s)|^2 +|\delta c_u^{\theta}(s)|^2|\delta {v}_{(s,e)} |^2 \Big\}\mathrm{d}\theta\tilde{N}(\mathrm{d}t,\mathrm{d}e)\bigg].\\
    \end{aligned}
  \end{equation*}
  On the right hand, the terms $b,\sigma,c$ are bounded by $C(1+|\delta v_{(t,e)}|^2)$. Because the above formula is integrable, we obtain that the term of $b,\sigma,c$ converge to $0$ as $(\epsilon_1,\epsilon_2)\rightarrow 0$ by Proposition \ref{prop433} as following and dominated convergence theorem. For the term of $\gamma$, which is bounded by $C(1+|\delta v_{(t,e)}|)$. If we set $(M,\mathscr{M})=([0,T]\times \mathcal{E}\times [0,1],\mathscr{B}([0,T])\otimes\mathscr{B}({\mathcal{E}})\otimes \mathscr{B}([0,1]))$, $K(\omega,A\times B\times C)=Leb(C)N(\omega,A\times B)$, then we obtain the following result by Proposition \ref{B.2.} in Appendix B,
  \begin{equation*}
    \int_0^t\int_{\mathcal{E}}\int_0^1\Big\{|\delta\gamma_x^{\theta}(s)|  +|\delta\gamma_{\tilde x}^{\theta}(s)| +|\delta\gamma_u^{\theta}(s)||\delta{v}_{(s,e)}| \Big\}\mathrm{d}\theta N(\mathrm{d}t,\mathrm{d}e)\rightarrow 0, \qquad (\epsilon_1,\epsilon_2)\rightarrow 0.
  \end{equation*}
  And we can have the term of $\gamma$ converges to $0$ by the usual dominated convergence theorem.
  Therefore, we have
  \begin{equation*}
    E\bigg[\sup_{0\leq t\leq T}|x_t^{\epsilon}-\hat{x}_t-x_t^1|^2\bigg]\leq C_{(\epsilon_1,\epsilon_2)}(\epsilon_1^2+\epsilon_2^2)=o(\epsilon_1^2+\epsilon_2^2),
  \end{equation*}
  where $C_{(\epsilon_1,\epsilon_2)}$ is a positive constant satisfying $C_{(\epsilon_1,\epsilon_2)}\rightarrow 0$ as ${(\epsilon_1,\epsilon_2)}\rightarrow 0$, which may vary line by line. Then the first estimate is proved. The proof of residual estimation is similar to Lemma 3.1 in \cite{Wu2011} and Lemma 1 in \cite{Xiao2011}.
\end{proof}

Now we define
\begin{equation}\label{hatJ}
  \begin{aligned}
    \hat{J}&=E\bigg[\int_0^T\int_{\mathcal{E}}\Big\{l_x(t)x_t^1+l_{\tilde x}(t)E[x_t^1]+l_y(t)y_t^1+l_{\tilde y}(t) E[y_t^1]+l_z(t)z_t^1+l_{\tilde z}(t)E[z_t^1]\\
    &+l_u(t)\hat{v}_t \Big\}\lambda(\mathrm{d}e)\mathrm{d}t+\int_0^T\int_{\mathcal{E}}\Big\{f_x(t)x_t^1+f_{\tilde x}(t)E[x_t^1]+f_y(t)y_t^1+f_{\tilde y}(t) E[y_t^1]+f_z(t)z_t^1\\
    &+f_{\tilde z}(t)E[z_t^1]+f_u(t)\hat{v}_t \Big\}N(\mathrm{d}t,\mathrm{d}e)+\varphi_x(T)x_T^1+\varphi_{\tilde x}(T)E[x_T^1] +\phi_y(\hat{y}_0)y_0^1+\sum_{i\geq 1}\psi_{\eta}(\tau_i,\hat{\eta}_i)\hat{\xi}_i\bigg].
  \end{aligned}
\end{equation}
Then we have the following lemma.
\begin{lemma1}
  Let Assumption {\rm (H)} hold. Then we have
  \begin{equation*}
    \lim_{(\epsilon_1,\epsilon_2)\rightarrow 0}\frac{1}{(\epsilon_1^2+\epsilon_2^2)^{\frac{1}{2}}}
    \Big(J(u^{\epsilon_1},\eta^{\epsilon_2})-J(\hat{u},\hat{\eta})-\hat{J}\Big)=0.
  \end{equation*}
\end{lemma1}
\begin{proof}
  It is easy to know that
  \begin{equation*}
    \begin{aligned}
      &J(u^{\epsilon_1},\eta^{\epsilon_2})-J(\hat{u},\hat{\eta})-\hat{J}\\
      &=E\bigg[\int_0^T\int_{\mathcal{E}}\bigg\{\int_0^1\Big\{l_x^{\theta}(t)\delta x^{\epsilon}_t +l_{\tilde x}^{\theta}(t)E[\delta x^{\epsilon}_t ]+l_y^{\theta}(t)\delta y^{\epsilon}_t+l_{\tilde y}^{\theta} (t)E[\delta y^{\epsilon}_t]+l_z^{\theta}(t)\delta z^{\epsilon}_t +l_{\tilde z}^{\theta}(t)E[\delta z^{\epsilon}_t ]\\
      &\qquad+\delta l_u^{\theta}(t)\hat{v}_t\Big\}\mathrm{d}\theta -l_x(t)x_t^1-l_{\tilde x}(t)E[x_t^1]-l_y(t)y_t^1-l_{\tilde y}(t) E[y_t^1]-l_z(t)z_t^1-l_{\tilde z}(t)E[z_t^1] \bigg\}\lambda(\mathrm{d}e)\mathrm{d}t\bigg]\\
      &+E\bigg[\int_0^T\int_{\mathcal{E}}\bigg\{\int_0^1\Big\{f_x^{\theta}(t)\delta x^{\epsilon}_{t-} +f_{\tilde x}^{\theta}(t)E[\delta x^{\epsilon}_{t-} ]+f_y^{\theta}(t)\delta y^{\epsilon}_{t-}+f_{\tilde y}^{\theta} (t)E[\delta y^{\epsilon}_{t-}]+f_z^{\theta}(t)\delta z^{\epsilon}_t \\
      &\qquad+f_{\tilde z}^{\theta}(t)E[\delta z^{\epsilon}_t ]+\delta f_u^{\theta}(t)\hat{v}_t\Big\}\mathrm{d}\theta -f_x(t)x_{t-}^1-f_{\tilde x}(t)E[x_{t-}^1]-f_y(t)y_{t-}^1-f_{\tilde y}(t) E[y_{t-}^1]-f_z(t)z_t^1\\
      &\qquad-f_{\tilde z}(t)E[z_t^1] \bigg\}N(\mathrm{d}t,\mathrm{d}e)\bigg]+E\bigg[\int_0^1 \Big(\varphi_x^{\theta}(T)\delta x_T^{\epsilon}+\varphi_{\tilde{x}}(T)E[\delta x_T^{\epsilon}] \Big)\mathrm{d}\theta-\varphi_x(t)x_t^1-\varphi_{\tilde x}E[x_t^1]\bigg] \\ &+E\bigg[\int_0^1\phi_y^{\theta}(0)\delta y_0^{\epsilon}\mathrm{d}\theta -\phi_y(\hat{y}_0) y_0^1\bigg] +E\bigg[\sum_{i\geq 1}\Big(\int_0^1\psi_{\eta}(\tau_i,\theta\eta^{\epsilon_2} +(1-\theta)\hat{\eta}_i) \hat{\xi}_i \mathrm{d}\theta -\psi_{\eta}(\tau_i,\hat{\eta}_i)\hat{\xi}_i\Big)\bigg].
    \end{aligned}
  \end{equation*}
  Since $l_u(t)$ is Lipschitz continuous with respect to $(\Lambda,\tilde \Lambda,u)$, we have
  \begin{equation*}
    \begin{aligned}
      &E\bigg[\int_0^T\int_{\mathcal{E}}\int_0^1\delta l_u^{\theta}(t)\mathrm{d}\theta\cdot \hat{v}_t \lambda(\mathrm{d} e)\mathrm{d}t\bigg]\\
      &\leq CE\bigg[\int_0^T\int_{\mathcal{E}}\Big\{|\delta x^{\epsilon}_t||\hat{v}_t|+E[|\delta x^{\epsilon}_t|] |\hat{v}_t|+|\delta y^{\epsilon}_t||\hat{v}_t|+E[|\delta y^{\epsilon}_t|] |\hat{v}_t|+|\delta z^{\epsilon}_t||\hat{v}_t|+E[|\delta z^{\epsilon}_t|] |\hat{v}_t|+|\hat{v}_t|^2 \Big\}\lambda(\mathrm{d}e) \mathrm{d}t\bigg]\\
      &\leq CE\bigg[\sup_{0\leq t\leq T}|\delta x^{\epsilon}_t|^2\bigg]^{\frac{1}{2}} E\bigg[\int_0^T\int_{\mathcal{E}}|\hat{v}_t|^2\lambda(\mathrm{d}e) \mathrm{d}t\bigg]^{\frac{1}{2}}
      +CE\bigg[\sup_{0\leq t\leq T}|\delta y^{\epsilon}_t|^2\bigg]^{\frac{1}{2}} E\bigg[\int_0^T\int_{\mathcal{E}}|\hat{v}_t|^2\lambda(\mathrm{d}e) \mathrm{d}t\bigg]^{\frac{1}{2}}\\
      &\qquad +CE\bigg[\int_0^T|\delta z^{\epsilon}_t|^2\mathrm{d}t\bigg]^{\frac{1}{2}} E\bigg[\int_0^T\int_{\mathcal{E}}|\hat{v}_t|^2\lambda(\mathrm{d}e) \mathrm{d}t\bigg]^{\frac{1}{2}}
      + CE\bigg[\int_0^T\int_{\mathcal{E}}|\hat{v}_t|^2\lambda(\mathrm{d}e) \mathrm{d}t\bigg]\\
      &\leq C_{(\epsilon_1,\epsilon_2)}(\epsilon_1^2+\epsilon_2^2)^{\frac{1}{2}}.
    \end{aligned}
  \end{equation*}
  We also know that
  \begin{equation}\label{99999}
    \begin{aligned}
      \int_0^1 l_x^{\theta}(t)\mathrm{d}\theta\cdot\delta x_t^{\epsilon}-l_x(t)x_t^1=
      \int_0^1\delta l_x^{\theta}(t)\mathrm{d}\theta\cdot x_t^1+\int_0^1l_x^{\theta}(t)\mathrm{d}\theta \cdot (\delta x_t^{\epsilon}-x_t^1).
    \end{aligned}
  \end{equation}
  For the first term of the right hand of \eqref{99999}, we have
  \begin{equation*}
    \begin{aligned}
      &E\bigg[\int_0^T\int_{\mathcal{E}}\int_0^1\delta l_x^{\theta}(t)\mathrm{d}\theta\cdot x_t^1 \lambda(\mathrm{d}e)\mathrm{d}t\bigg]\\
      &\leq C\bigg(E\bigg[\sup_{0\leq t\leq T}|x_t^1|^2 \bigg] \bigg)^{\frac{1}{2}}\bigg(E\bigg[\int_0^T\int_{\mathcal{E}}\int_0^1|\delta l_x^{\theta}(t)|^2\mathrm{d}\theta \lambda(\mathrm{d}e)\mathrm{d}t\bigg]\bigg)^{\frac{1}{2}}\\
      &\leq C\bigg(E\bigg[\sup_{0\leq t\leq T}|x_t^1|^2 \bigg] \bigg)^{\frac{1}{2}}\bigg( CE\bigg[\sup_{0\leq t\leq T}|\delta x_t^{\epsilon}|^2\bigg]+CE\bigg[\int_0^T \int_{\mathcal{E}}|\hat{v}_t|^2\lambda(\mathrm{d}e)\mathrm{d}t \bigg]\bigg)^{\frac{1}{2}}\\
      &\leq C_{(\epsilon_1,\epsilon_2)}(\epsilon_1^2+\epsilon_2^2)^{\frac{1}{2}}.
    \end{aligned}
  \end{equation*}
  For the second term of the right hand of \eqref{99999}, by Lemma \ref{lemma422} we have
  \begin{equation*}
    \begin{aligned}
      &E\bigg[\int_0^T\int_{\mathcal{E}}\int_0^1l_x^{\theta}(t)\mathrm{d}\theta \cdot (\delta x_t^{\epsilon}-x_t^1)\lambda(\mathrm{d}e)\mathrm{d}t\bigg]\\
      &\leq CE\bigg[\sup_{0\leq t\leq T}|x_t^{\epsilon}-\hat{x}_t-x_t^1|^2\bigg]^{\frac{1}{2}} \bigg(E\bigg[\int_0^T\int_{\mathcal{E}}\int_0^1 |l_x^{\theta}(t)|^2\mathrm{d}\theta \lambda(\mathrm{d}e)\mathrm{d}t\bigg]\bigg)^{\frac{1}{2}}\\
      &\leq C_{(\epsilon_1,\epsilon_2)}(\epsilon_1^2+\epsilon_2^2)^{\frac{1}{2}}.
    \end{aligned}
  \end{equation*}
  Thus we have
  \begin{equation*}
    \begin{aligned}
      E\bigg[\int_0^T\int_{\mathcal{E}}\Big\{\int_0^1 l_x^{\theta}(t)\mathrm{d}\theta\cdot\delta x_t^{\epsilon}-l_x(t)x_t^1\Big\}\lambda(\mathrm{d}e)\mathrm{d}t\bigg]\leq C_{(\epsilon_1,\epsilon_2)}(\epsilon_1^2+\epsilon_2^2)^{\frac{1}{2}}.
    \end{aligned}
  \end{equation*}
  By the similar argument, we also have the same result for the term of $l_{\tilde x}, l_y,l_{\tilde y},l_z, l_{\tilde z}$. Then we have
  \begin{equation*}
    \begin{aligned}
      &E\bigg[\int_0^T\int_{\mathcal{E}}\bigg\{\int_0^1\Big\{l_x^{\theta}(t)\delta x^{\epsilon}_t +l_{\tilde x}^{\theta}(t)E[\delta x^{\epsilon}_t ]+l_y^{\theta}(t)\delta y^{\epsilon}_t+l_{\tilde y}^{\theta} (t)E[\delta y^{\epsilon}_t]+l_z^{\theta}(t)\delta z^{\epsilon}_t +l_{\tilde z}^{\theta}(t)E[\delta z^{\epsilon}_t ]+\delta l_u^{\theta}(t)\hat{v}_t\Big\}\mathrm{d}\theta\\
      &\quad -l_x(t)x_t^1-l_{\tilde x}(t)E[x_t^1]-l_y(t)y_t^1-l_{\tilde y}(t) E[y_t^1]-l_z(t)z_t^1-l_{\tilde z}(t)E[z_t^1] \bigg\}\lambda(\mathrm{d}e)\mathrm{d}t\bigg]
      \leq C_{(\epsilon_1,\epsilon_2)}(\epsilon_1^2+\epsilon_2^2)^{\frac{1}{2}}.
    \end{aligned}
  \end{equation*}
  Similarly, we also have the same estimates of the term $f,\varphi,\phi,\psi$. Therefore, we have
  \begin{equation*}
    |J(u^{\epsilon_1},\eta^{\epsilon_2})-J(\hat{u},\hat{\eta})-\hat{J}|\leq C_{(\epsilon_1,\epsilon_2)}(\epsilon_1^2+\epsilon_2^2)^{\frac{1}{2}}.
  \end{equation*}
\end{proof}

\section{Adjoint equations and the maximum principle.}

Now we introduce the following adjoint equation,

\begin{equation}\label{adjoint}
  \left\{
    \begin{aligned}
      \mathrm{d}q_t&=\int_{\mathcal{E}}\Big\{q_tg_y(t)-l_y(t)-\check{\mathbb{E}}[f_y(t)] +E[q_tg_{\tilde{y}}(t)-l_{\tilde y}(t)-\check{\mathbb E}[f_{\tilde y}(t)]]\Big\}\lambda(\mathrm{d}e)\mathrm{d}t\\
      &\qquad+\int_{\mathcal{E}}\Big\{q_tg_{z}(t)-l_z(t)- \check{\mathbb{E}}[f_z(t)]+{E}[q_tg_{\tilde{z}}(t) -l_{\tilde z}(t)-\check{\mathbb E}[f_{\tilde z}(t)]] \Big\}\lambda(\mathrm{d}e)\mathrm{d}B_t,\\
      \mathrm{d}p_t&=-\int_{\mathcal{E}}\Big\{p_tb_x(t)+\theta_t\sigma_x(t)-q_tg_x+l_x(t)+(p_t +\vartheta_t) \check{\mathbb E}[\gamma_x(t)]+\vartheta_t\check{\mathbb E}[c_x(t)]\\
      &\qquad+\check{\mathbb E}[f_{ x}(t)]+{E}\big[p_tb_{\tilde x}(t)-q_tg_{\tilde x}(t)+\theta_t\sigma_{\tilde x}(t)+l_{\tilde x}(t)+(p_t +\vartheta_t)\check{\mathbb E}[\gamma_{\tilde{x}}(t)]\\
      &\qquad+\vartheta_t\check{\mathbb E}[c_{\tilde x} (t)]+\check{\mathbb E}[f_{\tilde x}(t)] \big] \Big\}\lambda(\mathrm{d}e)\mathrm{d}t
      +\theta_t\mathrm{d}B_t+\int_{\mathcal{E}}\vartheta_t \tilde{N}(\mathrm{d}t,\mathrm{d}e),\\
      q_0&=-\phi_y(\hat y_0),\qquad p_T=\varphi_x(T)+{E}[\varphi_{\tilde x}(T)]-h_x(T)q_T -E[h_{\tilde x}(T)q_T].
    \end{aligned}
  \right.
\end{equation}

In order to get the existence and uniqueness of the mean-field FBSDEP \eqref{adjoint}. Noticing that the mean-field FBSDEP \eqref{adjoint} is decoupled, we refer to Lemma 2.4 in \cite{Tang1994} and \cite{Ma2015}. Since the partial derivatives of the state coefficients are bounded, there exists a unique solution of \eqref{adjoint} i.e. $(q,p,\theta,\vartheta)\in S^2[0,T]\times S^2[0,T]\times M^2[0,T]\times F^2[0,T]$.

Applying It\^o's formula to $p_tx_t^1+q_ty_t^1$ 
and noting that
\begin{eqnarray*}
  \begin{aligned}
    {E}\bigg[\int_0^T\Big(p_tG_t+q_tH_t\Big)\mathrm{d}\hat{\xi}_t\bigg]={E}\bigg[\sum_{i\geq 1}(p_{\tau_i}G_{\tau_i}+q_{\tau_i}H_{\tau_i})(\xi_i-\hat{\eta}_i)\bigg],
  \end{aligned}
\end{eqnarray*}
we have
\begin{equation}\label{cost222}
  \begin{aligned}
    &{E}\bigg[\varphi_x(T)x_t^1+\varphi_{\tilde{x}}(T){E}[x_t^1]+\int_0^T\int_{\mathcal{E}} \Big\{ l_x(t)x_t^1+l_{\tilde x}(t){E}[x_t^1]+l_y(t)y_t^1+l_{\tilde y}(t){E}[y_t^1]+l_z(t)z_t^1\\
    &+l_{\tilde z}(t){E}[z_t^1]\Big\}\lambda(\mathrm{d}e)\mathrm{d}t+\int_0^T\int_{\mathcal{E}}\Big\{f_x(t)x_{t-}^1+f_{\tilde x}(t){E}[x_{t-}^1] +f_y(t)y_{t-}^1+f_{\tilde y}(t){E}[y_{t-}^1]
    +f_z(t)z_{t-}^1\\
    &+f_{\tilde z}(t){E}[z_{t-}^1]\Big\}N(\mathrm{d}t,\mathrm{d}e)+\phi_y(\hat{y}_0)y_0^1\bigg]={E}\bigg[\int_0^T\int_{\mathcal{E}}\Big(p_tb_u(t)+\theta_t\sigma_u(t)-q_tg_u(t)\Big)\hat{v}_t\lambda(\mathrm{d}e)\mathrm{d}t \bigg] \\
    &+{E}\bigg[\sum_{i\geq 1}(p_{\tau_i}G_{\tau_i}+q_{\tau_i}H_{\tau_i})(\xi_i-\hat{\eta}_i)+\int_0^T\int_{\mathcal{E}}\Big((p_{t-}+\vartheta_t)\gamma_u (t)+\vartheta_tc_u(t)\Big)\hat{v}_t N(\mathrm{d}t,\mathrm{d}e)\bigg].\\
  \end{aligned}
\end{equation}
By substituting \eqref{cost222} into \eqref{hatJ}, we get
\begin{equation*}
  \begin{aligned}
    \hat{J}&={E}\bigg[\epsilon_1\int_0^T\int_{\mathcal{E}} \mathcal H_u(t,\hat{x}_t,\hat{y}_t,\hat{z}_t,\hat{u}_{(t,e)}, p_t,\theta_t,q_t) (v_{(t,e)}-\hat{u}_{(t,e)})\lambda(\mathrm{d}e)\mathrm{d}t\bigg]\\
    &\qquad+{E}\bigg[\epsilon_1\int_0^T\int_{\mathcal{E}}\widetilde{\mathcal H}_u(t,\hat{x}_{t-},\hat{y}_{t-},\hat{z}_{t-},\hat{u}_{(t,e)} ,p_{t-} ,\vartheta_t,e)(v_{(t,e)}-\hat{u}_{(t,e)})N(\mathrm{d}t,\mathrm{d}e)\bigg]\\
    &\qquad+{E}\bigg[\epsilon_2\sum_{i\geq 1}(p_{\tau_i}G_{\tau_i}+q_{\tau_i}H_{\tau_i}+\psi_{\eta}(\tau_i,\hat{\eta}_i))(\xi_i-\hat{\eta}_i)\bigg],
  \end{aligned}
\end{equation*}
where $\mathcal H$ and $\widetilde{\mathcal H}$ satisfy the following equations associated with random variables $(x,y,z)$ belonging to $ L^1(\Omega,\mathscr{F},P)$.
\begin{equation*}
  \begin{aligned}
    \mathcal{H}(t,x,y,z,u,p,\theta,q,e)&=-qg(t,x,E[x],y,E[y],z,E[z],u,e)+pb(t,x,E[x],u,e)\\
    &\qquad+\theta\sigma(t,x,E[x],u,e)+l(t,x,E[x],y,E[y],z,E[z],u,e),\\
    \widetilde{\mathcal{H}}_u(t,x,y,z,u,p,\vartheta,e)&=p\gamma(t,x,E[x],u,e)+\vartheta\gamma(t,x,E[x],u,e)+\vartheta c(t,x,E[x],u,e)\\
    &\qquad+f(t,x,E[x],y,E[y],z,E[z],u,e).
  \end{aligned}
\end{equation*}
In order to obtain the necessary conditions of optimal control, we need to calculate the following limitation
\begin{equation*}
  \lim_{\epsilon_1\rightarrow 0}\lim_{\epsilon_2\rightarrow 0}\frac{1}{\epsilon_1}\hat{J}\qquad \text{and}\qquad \lim_{\epsilon_2\rightarrow 0}\lim_{\epsilon_1\rightarrow 0}\frac{1}{\epsilon_2}\hat{J}.
\end{equation*}
We first give the first part of the stochastic maximum principle.
\begin{equation*}
  \begin{aligned}
    \lim_{\epsilon_1\rightarrow 0}\lim_{\epsilon_2\rightarrow 0}\frac{1}{\epsilon_1}\hat{J}&={E}\bigg[\int_0^T\int_{\mathcal{E}}\mathcal{H}_u(t,\hat{x}_t,\hat{y}_t, \hat{z}_t,\hat{u}_{(t,e)},p_t,\theta_t,q_t,e) (v_{(t,e)}-\hat{u}_{(t,e)})\lambda(\mathrm{d}e)\mathrm{d}t\bigg]\\
    &\qquad+{E}\bigg[\int_0^T\int_{\mathcal{E}}\widetilde{\mathcal{H}}(t,\hat{x}_{t-},\hat{y}_{t-},\hat{z}_{t-},\hat{u}_{(t,e)},p_{t-} ,\vartheta_t,e)(v_{(t,e)}-\hat{u}_{(t,e)})N(\mathrm{d}t,\mathrm{d}e)\bigg].\\
  \end{aligned}
\end{equation*}
For the atom $(T_n,U_n)$ of Possion random measure, We define
\begin{equation*}
  [\![(T_n,U_n)]\!]\triangleq\big\{(\omega,t,e)\in\Omega\times[0,T]\times {\mathcal{E}} \ |\ (T_n(\omega),U_n(\omega))=(t,e)\big\},
\end{equation*}
which is the graph of $(T_n,U_n)$. Similar to the Lemma 3.3 in Song \cite{Song2021}, we have the following lemma. The proof can be seen in the Appendix C. 
\begin{lemma1}\label{lemma551}
  For any $n\geq 1$, we have $[\![(T_n,U_n)]\!]\in\mathscr{G}\otimes \mathscr{B}({\mathcal{E}})$.
\end{lemma1}

Since $(T_n,U_n)$ is the atom of Possion random measure, it also is the atom of measure $\mu$. That is
\begin{equation*}
  \begin{aligned}
    \mu\Big(\Big(\bigcup_{n=1}^{\infty}[\![(T_n,U_n)]\!]\Big)^c\Big)&=0,\\
    \nu\Big(\bigcup_{n=1}^{\infty}[\![(T_n,U_n)]\!]\Big)=E\bigg[\sum_{n=1}^{\infty}\int_0^T \int_E &\mathbbm{1}_{[\![(T_n,U_n)]\!]}\lambda(\mathrm{d}e)\mathrm{d}t\bigg]=0.
  \end{aligned}
\end{equation*}
Therefore, $\mu$, $\nu$ are singular measures (remember $\nu=\mathbb{P}\times Leb\times \lambda$). If some property is $\mu$-a.s. true, it means that this property only can be used to characterizing features on $O:=\bigcup_{n=1}^{\infty}[\![(T_n,U_n)]\!]$. Similarly, if some property is $\nu$-a.s. true, it means that this property only can be used to characterizing features on $O^c$. For any $A\in\mathscr{G}$, any $V\in {U}$, define
\begin{equation*}
  v_{(t,e)}=
  \left\{
    \begin{aligned}
      &\hat{u},\qquad &&\text{if  } (t,\omega,e)\in O;\\
      &V,\qquad &&\text{if  } (t,\omega,e)\in A\cap O^c;\\
      &\hat{u},\qquad &&\text{if  } (t,\omega,e)\in A^c\cap O^c.
    \end{aligned}
  \right.
\end{equation*}
Then we can know that $v\in\mathcal{U}_{ad}$ by Lemma \ref{lemma551}, we have
\begin{equation*}
  \lim_{\epsilon_1\rightarrow 0}\lim_{\epsilon_2\rightarrow 0}\frac{1}{\epsilon_1}\hat{J}={E}\bigg[\int_0^T\int_{\mathcal{E}}\mathbbm{1}_A\mathcal{H}_u(t,\hat{x}_t,\hat{y}_t, \hat{z}_t,\hat{u}_{(t,e)},p_t,\theta_t,q_t,e)(V-\hat{u}_{(t,e)})\lambda(\mathrm{d}e)\mathrm{d}t\bigg]\geq 0,
\end{equation*}
for any $A\in\mathscr{G}$, any $V\in {U}$.
Therefore we obtain the following maximum principle, which called the continuous part of maximum principle
\begin{equation*}
  \mathcal{H}_u(t,\hat{x}_t,\hat{y}_t,\hat{z}_t,\hat{u}_{(t,e)},p_t,\theta_t,q_t,e)(V-\hat{u}_{(t,e)})\geq 0,\qquad \nu\text{-a.s.}
\end{equation*}
is true for any $V\in U$. For any $A\in\mathscr{G}$, any $V\in {U}$, define
\begin{equation*}
  v_{(t,e)}=
  \left\{
    \begin{aligned}
      &\hat{u},\qquad &&\text{if } (t,\omega,e)\in O^c;\\
      &V,\qquad &&\text{if } (t,\omega,e)\in A\cap O;\\
      &\hat{u},\qquad &&\text{if } (t,\omega,e)\in A^c\cap O.
    \end{aligned}
  \right.
\end{equation*}
Then we can know that $v\in\mathcal{U}_{ad}$ by Lemma \ref{lemma551}, we have
\begin{equation*}
  \begin{aligned}
    \lim_{\epsilon_1\rightarrow 0}\lim_{\epsilon_2\rightarrow 0}\frac{1}{\epsilon_1}\hat{J}={E}\bigg[\int_0^T\int_{\mathcal{E}}\mathbbm{1}_A\widetilde{\mathcal{H}}_u(t,\hat{x}_{t-},\hat{y}_{t-},\hat{z}_{t-}, \hat{u}_{(t,e)},p_{t-} ,\vartheta_t,e)(V-\hat{u}_{(t,e)})N(\mathrm{d}t,\mathrm{d}e)\bigg]\geq 0,\\
  \end{aligned}
\end{equation*}
for any $A\in\mathscr{G}$, any $V\in {U}$.
Therefore we obtain the following maximum principle, which called the jump part of maximum principle
\begin{equation*}
  \widetilde{\mathcal{H}}_u(t,\hat{x}_{t-},\hat{y}_{t-},\hat{z}_{t-}, \hat{u}_{(t,e)},p_{t-} ,\vartheta_t,e)(V-\hat{u}_t)\geq 0,\qquad \mu\text{-a.s.}
\end{equation*}
is true for any $V\in U$. Next we give the second part of the stochastic maximum principle.
\begin{equation*}
  \begin{aligned}
    \lim_{\epsilon_2\rightarrow 0}\lim_{\epsilon_1\rightarrow 0}\frac{1}{\epsilon_2}\hat{J}={E}\bigg[\sum_{i\geq 1}(p_{\tau_i}G_{\tau_i}+q_{\tau_i}H_{\tau_i} +\psi_{\eta}(\tau_i,\hat{\eta}_i))(\xi_i-\hat{\eta}_i)\bigg]\geq 0.
  \end{aligned}
\end{equation*}

Then we give the main theorem of this paper.
\begin{theorem}\label{necessary}
  Let Assumption {\rm (H)} hold. Suppose that $(\hat{u},\hat{\eta})$ is the optimal control, $(\hat{x},\hat{y},\hat{z},\hat{k})$ is the corresponding trajectory, and $(q,p,\theta,\vartheta)\in S^2[0,T]\times S^2[0,T]\times M^2[0,T]\times F^2[0,T]$ satisfies \eqref{adjoint}. Then for any $V\in{U}$, we have
  \begin{flalign}
      {\mathcal{H}}_u(t,\hat{x}_t,\hat{y}_t,\hat{z}_t,\hat{u}_{(t,e)},p_t,\theta_t,q_t,e)(V-\hat{u}_{(t,e)})\geq 0,\qquad \nu \text{-a.s.} \label{contin}\\
      \widetilde{{\mathcal{H}}}_u(t,\hat{x}_{t-},\hat{y}_{t-},\hat{z}_{t-}, \hat{u}_{(t,e)},p_{t-} ,\vartheta_t,e)(V-\hat{u}_{(t,e)})\geq 0,\qquad \mu\text{-a.s.}\label{jump}
  \end{flalign}
  and for any $\xi\in\mathcal{K}_{ad}$, we have
  \begin{equation}\label{impulse}
    {E}\bigg[\sum_{i\geq 1}(p_{\tau_i}G_{\tau_i}+q_{\tau_i}H_{\tau_i} +\psi_{\eta}(\tau_i,\hat{\eta}_i))(\xi_i-\hat{\eta}_i)\bigg]\geq 0.
\end{equation}
\end{theorem}

\begin{remark1}
  {\rm i)} Compared with the results of classical predictable framework, \eqref{jump} is used to be characterize the optimal control at the jump time of Poisson random measure. On the other hand, \eqref{contin} looks similar to the classical result. But they are different because of the difference of dual equations. \eqref{impulse} is used to characterize optimal impulse control.\\
  {\rm ii)} The stochastic maximum principle \eqref{contin} and \eqref{jump} seem contradictory, but they hold almost everywhere under the mutually singular measures $\nu$ and $\mu$, respectively. So they won't happen at the same time.
\end{remark1}

Further, we can derive the following useful corollary.
\begin{corollary1}\label{corollary}
  Under the same assumptions of Theorem \ref{necessary}, we have
  \begin{equation}
    p_{\tau_i}G_{\tau_i}+q_{\tau_i}H_{\tau_i} +\psi_{\eta}(\tau_i,\hat{\eta}_i)=0, \qquad i\geq 1,\ a.s.
  \end{equation}
\end{corollary1}
\begin{proof}
  Similar to Corollary 3.1 in \cite{Wu2011}, it is easy to see that
  \begin{equation*}
    E\bigg[\sum_{i\geq 1}|p_{\tau_i}G_{\tau_i}+q_{\tau_i}H_{\tau_i} +\psi_{\eta}(\tau_i,\hat{\eta}_i)|^2\bigg]<\infty.
  \end{equation*}
  Define $M_i=p_{\tau_i}G_{\tau_i}+q_{\tau_i}H_{\tau_i} +\psi_{\eta}(\tau_i,\hat{\eta}_i)$. Let $\xi\in\mathcal{K}_{ad}$ be such that
  \begin{equation*}
    \xi_i=\left\{
    \begin{aligned}
      &\hat{\eta}_i-M_i,\qquad \text{if } M_i\geq 0,\\
      &\hat{\eta}_i,\qquad\qquad\ \,  \text{otherwise}.
    \end{aligned}
    \right.
  \end{equation*}
  Then it follows from \eqref{impulse} that
  \begin{equation*}
    E\bigg[\sum_{i\geq 1}(M_i)^2\mathbbm{1}_{\{M_i\geq 0\}}\bigg]=0.
  \end{equation*}
  Similarly, define $\xi\in\mathcal{K}_{ad}$ by
  \begin{equation*}
    \xi_i=\left\{
    \begin{aligned}
      &\hat{\eta}_i-M_i,\qquad \text{if } M_i\leq 0,\\
      &\hat{\eta}_i,\qquad\qquad\ \,  \text{otherwise}.
    \end{aligned}
    \right.
  \end{equation*}
  Then
  \begin{equation*}
    E\bigg[\sum_{i\geq 1}(M_i)^2\mathbbm{1}_{\{M_i\leq 0\}}\bigg]=0.
  \end{equation*}
  So we have $E[\sum_{i\geq 1}(M_i)^2]=0$, then the result follows.
\end{proof}

 Next we give the sufficient condition for optimal control.
\begin{theorem}\label{sufficient}
  Let Assumption {\rm (H)} hold. Assume that $\varphi, \phi, \psi, h, {\mathcal{H}}, \widetilde{{\mathcal{H}}}$ are convex with respect to $(x,y,z,u,\eta)$. Suppose that $({u},{\eta})\in\mathcal{A}$ is a given control, $({x},{y},{z}, k)$ is the corresponding solution of \eqref{state}, $(q,p,\theta,\vartheta)$ is the corresponding solution of \eqref{adjoint}. 
  Then $(u,\eta)$ is an optimal control if it satisfies \eqref{contin}-\eqref{impulse}.
\end{theorem}
\begin{proof}
  Let $(v,\xi)$ be an any admissible control, then we have
  \begin{equation*}
  \begin{aligned}
    J({u},{\eta})-J(v,\xi)&=E\bigg[\int_0^T\int_{\mathcal{E}}\delta l(t)\lambda(\mathrm{d}e)\mathrm{d}t +\int_0^T\int_{\mathcal{E}}\delta f(t)N(\mathrm{d}t,\mathrm{d}e)+\varphi({x}^{u,\eta}_T,E[{x}^{u,\eta}_T]) \\
    &\qquad-\varphi({x}^{v,\xi}_T,E[ {x}^{v,\xi}_T])+\phi({y}^{u,\eta}_0) -\phi(y^{v,\xi}_0) +\sum_{i\geq 1}(\psi(\tau_i,{\eta_i})-\psi(\tau_i,\xi_i))\bigg],
  \end{aligned}
  \end{equation*}
  where $\delta \kappa(t)=\kappa(t,{x}^{u,\eta}_t,E[{x}^{u,\eta}_t],{y}^{u,\eta}_t,E[{y}^{u,\eta}_t] ,{z}^{u,\eta}_t, E[{z}^{u,\eta}_t], u_{(t,e)},e)-\kappa(t,{x}^{v,\xi}_t,E[{x}^{v,\xi}_t], {y}^{v,\xi}_t, E[{y}^{v,\xi}_t]$, ${z}^{v,\xi}_t, E[{z}^{v,\xi}_t]$, $v_{(t,e)},e)$ with $\kappa=l,f$. From the convexity of $\varphi$, $\phi$ and $\psi$, we have
  \begin{equation*}
    \begin{aligned}
      &E\bigg[\varphi({x}^{u,\eta}_T,E[{x}^{u,\eta}_T])-\varphi({x}^{v,\xi}_T,E[ {x}^{v,\xi}_T])+\phi({y}^{u,\eta}_0) -\phi(y^{v,\xi}_0) +\sum_{i\geq 1}(\psi(\tau_i,{\eta_i})-\psi(\tau_i,\xi_i))\bigg]\\
      &\leq E\bigg[\varphi_x({x}^{u,\eta}_T,E[{x}^{u,\eta}_T])({x}^{u,\eta}_t-x^{v,\xi}_t)+\varphi_{\tilde{x}} ({x}^{u,\eta}_T,E[{x}^{u,\eta}_T])E[{x}^{u,\eta}_t-x^{v,\xi}_t]\\
      &\qquad\qquad\qquad\qquad\qquad\quad+\phi_y({y}^{u,\eta}_0)({y}^{u,\eta}_0-y^{v,\xi}_0) +\sum_{i\geq 1}\psi_{\eta}(\tau_i, {\eta}_i)({\eta}_i-\xi_i) \bigg].
    \end{aligned}
  \end{equation*}
  Applying It\^o's formula to $p_t({x}^{u,\eta}_t-x^{v,\xi}_t)+q_t({y}^{u,\eta}_t-y^{v,\xi}_t)$, noticing that
  \begin{eqnarray*}
    \begin{aligned}
      {E}\bigg[\int_0^T\Big(p_tG_t+q_tH_t\Big)\mathrm{d}(\eta_t-\xi_t)\bigg]={E}\bigg[\sum_{i\geq 1}(p_{\tau_i}G_{\tau_i}+q_{\tau_i}H_{\tau_i})({\eta}_i-\xi_i)\bigg].
    \end{aligned}
  \end{eqnarray*}
  Since ${\mathcal{H}},\widetilde{{\mathcal{H}}}$ are convex function, we have
  \begin{equation*}
    \begin{aligned}
      &J({u},{\eta})-J(v,\xi)\\
      &\leq E\bigg[\int_0^T\int_{\mathcal{E}}\Big\{\delta \mathcal{H}(t)- \mathcal{H}_x(t)(x_t^{u,\eta}-x^{v,\xi}_t)-\mathcal{H}_{\tilde x}(t)E[x_t^{u,\eta}-x^{v,\xi}_t]-\mathcal{H}_y(t)(y_t^{u,\eta}-y^{v,\xi}_t)\\
      &\qquad-\mathcal{H}_{\tilde y}(t)E[y_t^{u,\eta}-y^{v,\xi}_t] -\mathcal{H}_z(t)(z_t^{u,\eta}-z^{v,\xi}_t)-\mathcal{H}_{\tilde z}(t)E[z_t^{u,\eta}-z^{v,\xi}_t]\Big\}\lambda(\mathrm{d}e)\mathrm{d}t\\
      &\qquad+\int_0^T\int_{\mathcal{E}}\Big\{\delta \widetilde{\mathcal{H}}(t)- \widetilde{\mathcal{H}}_x(t)(x_{t-}^{u,\eta}-x^{v,\xi}_{t-})-\widetilde{\mathcal{H}}_{\tilde x}(t)E[x_{t-}^{u,\eta}-x^{v,\xi}_{t-}]
      -\widetilde{\mathcal{H}}_y(t)(y_{t-}^{u,\eta}-y^{v,\xi}_{t-})\\
      &\qquad-\widetilde{\mathcal{H}}_{\tilde y}(t)E[y_{t-}^{u,\eta}-y^{v,\xi}_{t-}]-\widetilde{\mathcal{H}}_z(t)(z_t^{u,\eta}-z^{v,\xi}_t)
      -\widetilde{\mathcal{H}}_{\tilde z}(t)E[z_t^{u,\eta}-z^{v,\xi}_t]\Big\}N(\mathrm{d}t,\mathrm{d}e)\bigg]\\
      &\qquad +E\bigg[\sum_{i\geq 1}(p_{\tau_i}G_{\tau_i}+q_{\tau_i}H_{\tau_i}+\psi_{\eta}(\tau_i, {\eta}_i))({\eta}_i-\xi_i)\bigg]\\
      &\leq E\bigg[\int_0^T\int_{\mathcal{E}}\mathcal{H}_u(t)(u_t-v_t)\lambda(\mathrm{d}e)\mathrm{d}t\bigg]
      +E\bigg[\int_0^T\int_{\mathcal{E}}\widetilde{\mathcal{H}}_u(t)(u_t-v_t)N(\mathrm{d}t,\mathrm{d}e)\bigg]\\
      &\qquad +E\bigg[\sum_{i\geq 1}(p_{\tau_i}G_{\tau_i}+q_{\tau_i} H_{\tau_i}+\psi_{\eta}(\tau_i, {\eta}_i))({\eta}_i-\xi_i)\bigg]
      \leq 0,
    \end{aligned}
  \end{equation*}
  where $\delta \kappa(t)=\kappa(t,{x}^{u,\eta}_t,{y}^{u,\eta}_t,{z}^{u,\eta}_t,{u}_{(t,e)},p_t, \theta_t,q_t,e)-\kappa(t,{x}^{v,\xi}_t,{y}^{v,\xi}_t,{z}^{v,\xi}_t,{v}_{(t,e)},p_t,\theta_t,q_t,e)$ and $\kappa_\varsigma(t)$ $=\kappa_\varsigma(t,{x}^{u,\eta}_t,{y}^{u,\eta}_t,{z}^{u,\eta}_t,{u}_{(t,e)},p_t, \theta_t,q_t,e)$ with $\kappa=\mathcal{H},\widetilde{\mathcal{H}}$ and $\varsigma=x,\tilde{x},y,\tilde{y},z,\tilde{z},u$. This implies that $(u,\eta)$ is the optimal control.
\end{proof}

\section{Application in LQ case}
In this section, we present two examples. One is a general LQ example to show how to find the optimal control and the related feedback representation form by our results. The other is a special example to illustrate that compared with predictable control, the progressive optimal control can indeed make the cost functional get smaller value in some cases. 
Now, we first introduce the corresponding dynamic of the LQ problem and consider the following
MF-FBSDEP
\begin{equation}
  \left\{
    \begin{aligned}
      \mathrm{d}x_t&=\int_{\mathcal{E}}\Big\{A_{1,t}x_{t}+B_{1,t}E[x_{t}]+C_{1,t}u_{(t,e)}\Big\}\lambda(\mathrm{d}e) \mathrm{d}t+\int_{\mathcal{E}}\Big\{A_{2,t}x_{t}+B_{2,t}E[x_{t}]\\
      &\qquad+C_{2,t}u_{(t,e)}\Big\}\lambda(\mathrm{d}e) \mathrm{d}B_t+\int_{\mathcal{E}}\Big\{A_{3,t}x_{t-}+B_{3,t} E[x_{t-}]+C_{3,t}u_{(t,e)}\Big\}N(\mathrm{d}t, \mathrm{d}e)\\
      &\qquad +\int_{\mathcal{E}}\Big\{A_{4,t}x_{t-}+B_{4,t} E[x_{t-}]+C_{4,t}u_{(t,e)}\Big\}\tilde N(\mathrm{d}t, \mathrm{d}e),\\
      \mathrm{d}y_t&=-\int_{\mathcal{E}}\Big\{A_{5,t}x_t+B_{5,t}E[x_t]+F_{1,t}y_t+\tilde F_{1,t}E[y_t]+K_{t} z_t -K_{t}E[z_t]\\
      &\qquad+C_{5,t}u_{(t, e)}\Big\}\lambda(\mathrm{d}e)\mathrm{d}t+z_t\mathrm{d}B_t+\int_{\mathcal{E}}k_t\tilde N(\mathrm{d}t,\mathrm{d}e)+H_t\mathrm{d}\eta_t,\\
      x_0&=x_0,\qquad y_T=Mx_T.
    \end{aligned}
  \right.
\end{equation}
Here, all coefficients $F_1,\bar{F}_1, A_i,B_i,C_i, i=1,2,\cdots,5$ are deterministic functions satisfying the Assumption (H1) and $(u,\eta)\in \mathcal{A}$ with $U=R$. The cost functional is
\begin{equation}
  \begin{aligned}
    J(u,\eta)&=E\bigg[\int_0^T\int_{\mathcal{E}}\Big\{A_{6,t}x_t+B_{6,t}E[x_t]+F_{2,t}y_t+\tilde F_{2,t}E[y_t]+\frac{1}{2}u^2_{(t, e)}\Big\}\lambda(\mathrm{d}e)\mathrm{d}t+\int_0^T\int_{\mathcal{E}}\Big\{A_{7,t}x_{t-} \\
    &\qquad+B_{7,t}E[x_{t-}]+F_{3,t}y_{t-} +\tilde F_{3,t}E[y_{t-}]  +\frac{1}{2}u^2_{(t, e)}\Big\}N(\mathrm{d}t,\mathrm{d}e)+\frac{\delta}{2}|E[x_T]|^2+\frac{1}{2}y_0^2+\frac{1}{2} \sum_{i\geq 1}\eta_i^2\bigg],
  \end{aligned}
\end{equation}
where $\delta>0$ and all coefficients $A_i,B_i,F_j,\bar{F}_j,i=6,7,j=2,3$ are deterministic functions satisfying the Assumption (H2). We now use the Theorem \ref{necessary} and Corollary \ref{corollary} to solve the Problem (OCP).
From the necessary stochastic maximum principle, we know that if $(\hat{u},\hat{\eta})$ is optimal control, then it must satisfy
\begin{flalign}
    \hat{u}_{(t,e)}&=q_tC_{5,t}-p_tC_{1,t}-\theta_tC_{2,t},\qquad \nu\text{-a.s.} \label{3388}\\
    \hat{u}_{(t,e)}&=-p_{t-}C_{3,t}-\vartheta_tC_{3,t}-\vartheta_tC_{4,t},\qquad \mu\text{-a.s.} \label{3399}\\
    \hat{\eta}_i&=-q_{\tau_i}H_{\tau_i},\qquad P\text{-a.s.},\label{4400}
\end{flalign}
where $(q,p,\theta,\vartheta)$ are adjoint processes given by the following MF-FBSDEP
\begin{equation}\label{FBSDE111}
  \left\{
    \begin{aligned}
      \mathrm{d}q_t&=\int_{\mathcal{E}}\Big\{F_{1,t}q_t-F_{2,t}-F_{3,t}+\tilde F_{1,t}E[q_t]-\tilde F_{2,t} -\tilde F_{3,t}\Big\}\lambda(\mathrm{d}e)\mathrm{d}t +\int_{\mathcal{E}}\Big\{K_{t}q_t-K_{t} E[q_t]\Big\}\lambda(\mathrm{d}e)\mathrm{d}B_t,\\
      \mathrm{d}p_t&=-\int_{\mathcal{E}}\Big\{p_tA_{1,t}+\theta_tA_{2,t}-q_tA_{5,t}+A_{6,t}+(p_t+\vartheta_t) A_{3,t}+\vartheta_tA_{4,t}+A_{7,t}+E[p_tB_{1,t}-q_tB_{5,t}\\
      &\qquad+\theta_tB_{2,t}+B_{6,t}+(p_t+ \vartheta_t)B_{3,t}+\vartheta_tB_{4,t}+B_{7,t}]\Big\}\lambda(\mathrm{d}e)\mathrm{d}t+\theta_t\mathrm{d}B_t+\int_{\mathcal{E}}\vartheta_t\tilde{N}(\mathrm{d}t,\mathrm{d}e),\\
      q_0&=-\hat{y}_0,\qquad p_T=\delta E[\hat{x}_T]-Mq_T.
    \end{aligned}
  \right.
\end{equation}
Then we show that $(\hat{u},\hat{\eta})$ given by \eqref{3388}-\eqref{4400} is indeed the optimal control. Firstly, for any admissible control $(u,\eta)\in \mathcal{A}$, the total variation of the cost is
\begin{equation*}
  \begin{aligned}
    J(u,\eta)-J(\hat{u},\hat{\eta})=\frac{1}{2}E\bigg[\int_0^T\int_{\mathcal{E}} \big(u_{(t,e)}&-\hat{u}_{(t,e)}\big)^2\mathrm{d}t+\int_0^T\int_{\mathcal{E}}\big(u_{(t,e)}-\hat{u}_{(t,e)}\big)^2N(\mathrm{d}t,\mathrm{d}e)\\
     &+{\delta}\big(E[x_T-\hat{x}_T]\big)^2+\big(y_0-\hat{y}_0\big)^2+\sum_{i\geq 1}\big(\eta_i-\hat{\eta}_i\big)^2\bigg]+II,
  \end{aligned}
\end{equation*}
where
\begin{equation*}
    \begin{aligned}
      II&=E\bigg[\int_0^T\int_\mathcal{E}\Big\{A_{6,t}(x_t-\hat{x}_t)+B_{6,t}E[x_t-\hat{x}_t]+F_{2,t}(y_t-\hat{y}_t)+\bar{F}_{2,t}E[y_t-\hat{y}_t]+\hat{u}_{(t,e)}({u}_{(t,e)}-\hat{u}_{(t,e)})\Big\}\lambda(\mathrm{d}e)\mathrm{d}t\bigg]\\
      &\quad +E\bigg[\int_0^T\int_\mathcal{E}\Big\{A_{7,t}(x_{t-}-\hat{x}_{t-})+B_{7,t}E[x_{t-}-\hat{x}_{t-}]+F_{3,t}(y_{t-}-\hat{y}_{t-})+\bar{F}_{3,t}E[y_{t-}-\hat{y}_{t-}]\\
      &\qquad\quad+\hat{u}_{(t,e)}({u}_{(t,e)}-\hat{u}_{(t,e)})\Big\}N(\mathrm{d}t,\mathrm{d}e)+{\delta}E[\hat{x}_T](E[x_T-\hat{x}_T])+\hat{y}_0(y_0-\hat{y}_0)+\sum_{i\geq 1}\hat{\eta}_i(\eta_i-\hat{\eta}_i)\bigg].\\
    \end{aligned}
\end{equation*}
Nextly, we would like to prove that $II=0$. Applying It\^o's formula to $p_t(x_t-\hat{x}_t)+q_t(y_t-\hat{y}_t)$, we have
\begin{equation*}
    \begin{aligned}
      &E\Big[\delta E[\hat{x}_T](E[x_T-\hat{x}_T])+\hat{y}_0(y_0-\hat{y}_0)\Big]\\&=E\bigg[\int_0^T\int_\mathcal{E}\Big\{[C_{1,t}p_t+C_{4,t}p_t-C_{5,t}q_t](u_{(t,e)}-\hat{u}_{(t,e)})-A_{6,t}(x_t-\hat{x}_t)-B_{6,t}E[x_t-\hat{x}_t]-F_{2,t}(y_t-\hat{y}_t)\\
      &\qquad-\bar{F}_{2,t}E[y_t-\hat{y}_t]\Big\}\lambda(\mathcal{E})\mathrm{d}t +\int_0^T\int_\mathcal{E}\Big\{[C_{3,t}(p_{t-}+\vartheta_t)+C_{4,t}\vartheta_t](u_{(t,e)}-\hat{u}_{(t,e)})-A_{7,t}(x_{t-}-\hat{x}_{t-})\\
      &\qquad-B_{7,t}E[x_{t-}-\hat{x}_{t-}]-F_{3,t}(y_{t-}-\hat{y}_{t-})-\bar{F}_{3,t}E[y_{t-}-\hat{y}_{t-}]\Big\}N(\mathrm{d}t,\mathrm{d}e)+\int_0^Tq_tH_t\mathrm{d}(\eta_t-\hat{\eta}_t) \bigg],
    \end{aligned}
\end{equation*}
which implies that
\begin{equation*}
    \begin{aligned}
      II&=E\bigg[\int_0^T\int_\mathcal{E}\Big[C_{1,t}p_t+C_{4,t}p_t-C_{5,t}q_t+\hat{u}_{(t,e)}\Big](u_{(t,e)}-\hat{u}_{(t,e)})\lambda(\mathrm{d}e)\mathrm{d}t+\sum_{i\geq 1}(q_{\tau_i}H_{\tau_i}+\hat{\eta}_i)(\eta_i-\hat{\eta}_i)\\
      &\qquad+\int_0^T\int_\mathcal{E}\Big[C_{3,t}(p_{t-}+\vartheta_t)+C_{4,t}\vartheta_t+\hat{u}_{(t,e)}\Big](u_{(t,e)}-\hat{u}_{(t,e)})N(\mathrm{d}t,\mathrm{d}e)\bigg]=0.
    \end{aligned}
\end{equation*}
Therefore, $J(u,\eta)-J(\hat{u},\hat{\eta})\geq 0$, $(\hat{u},\hat{\eta})$ is indeed the optimal control. In addition, we want to get the feedback of optimal control $(\hat{u},\hat{\eta})$.
By the classical SDE theorem, we can get the following explicit solution of the first equation of FBSDE \eqref{FBSDE111}.
\begin{equation}\label{qqqt}
  q_t=-\hat{y}_0 e^{\int_0^t(F_{1,s}+\bar{F}_{1,s})\lambda(\mathcal{E}) \mathrm{d}s}-\int_0^t(F_{2,s}+\bar{F}_{2,s}+F_{3,s}+\bar{F}_{3,s})e^{\int_s^t( F_{1,r}+\bar{F}_{1,r})\lambda(\mathcal{E})\mathrm{d}r}\mathrm{d}s.
\end{equation}
Moreover, noticing the second equation of \eqref{FBSDE111} is a BSDEP with $q$, when we get the explicit express of $q$, it has unique adapted solution $(p,\theta,\vartheta)\in S^2[0,T]\times M^2[0,T]\times F^2[0,T]$ by Lemma \ref{BSDEsolution}. In order to get the explicit form of solution $(p,\theta,\vartheta)$, we define the following ordinary differential equation (ODE)
\begin{equation}
  \left\{
    \begin{aligned}
      \mathrm{d}p_t&=-\int_{\mathcal{E}}\Big\{\pi_tp_t+\Delta_t\Big\}\lambda(\mathrm{d}e)\mathrm{d}t,\\
      p_T&=\delta E[\hat{x}_T]-Mq_T,
    \end{aligned}
  \right.
\end{equation}
where $\pi_t=A_{1,t}+A_{3,t}+B_{1,t}+B_{3,t}$ and $\Delta_t=A_{6,t}+A_{7,t}+B_{6,t}+B_{7,t}-q_t(A_{5,t}+B_{5,t})$. We can know that $(p,0,0)$ is the unique solution of the BSDE in \eqref{FBSDE111} by classical BSDE theory.
Let $p_t=\Pi_tE[\hat{x}_t]+\Sigma_t$, and apply It\^o's formula to $p_t=\Pi_tE[\hat{x}_t]+\Sigma_t$, we have
\begin{equation}
    \begin{aligned}
      \mathrm{d}p_t=\int_{\mathcal{E}}\Big\{\Big(\frac{\dot{\Pi}_t}{\lambda(\mathcal{E})}&+\Pi_t\pi_t-\Pi_t^2(C_{1,t}^2+C_{3,t}^2)\Big)E[\hat{x}_t]+\Pi_tC_{1,t} C_{5,t}q_t\\
      &+\frac{\dot{\Sigma}_t}{\lambda(\mathcal{E})}-\Pi_t(C_{1,t}^2+C_{3,t}^2)\Sigma_t\Big\}\lambda(\mathrm{d}e)\mathrm{d}t.
    \end{aligned}
\end{equation}
Comparing the coefficients of above equation with the second equation of \eqref{FBSDE111}, we have the following Riccati equation
\begin{equation}
    \left\{
        \begin{aligned}
          &\dot{\Pi}_t+2\lambda(\mathcal{E})\pi_t\Pi_t-\lambda(\mathcal{E})(C_{1,t}^2+C_{3,t}^2)\Pi_t^2=0,\\
          &\Pi_T=\delta,
        \end{aligned}
    \right.
\end{equation}
and ODE
\begin{equation}
    \left\{
        \begin{aligned}
          &\dot{\Sigma}_t+\big[\pi_t-(C_{1,t}^2+C_{3,t}^2)\Pi_t\big]\lambda(\mathcal{E})\Sigma_t+\Pi_tC_{1,t}C_{5,t}q_t\lambda(\mathcal{E})+\Delta_t\lambda(\mathcal{E})=0,\\
          &\Sigma_T=-Mq_T,
        \end{aligned}
    \right.
\end{equation}
 Solving above two equations, we have
\begin{equation}\label{PPPt}
    \Pi_t=\frac{1}{\delta^{-1}e^{-2\int_t^T\pi_s\lambda(\mathcal{E})\mathrm{d}s}+\int_t^T\lambda(\mathcal{E})(C_{1,t}^2+C_{3,t}^2)e^{-2\int_t^s\pi_r\lambda(\mathcal{E})\mathrm{d}r}\mathrm{d}s},
\end{equation}
and
\begin{equation}\label{SSSt}
    \Sigma_t=-Mq_Te^{\int_t^T[\pi_s-(C_{1,s}^2+C_{3,s}^2)\Pi_s]\lambda(\mathcal{E})\mathrm{d}s}+\int_t^T(\Pi_sC_{1,s}C_{5,s}q_s+\Delta_s)e^{-\int_t^s[(C_{1,r}^2+C_{3,r}^2)\Pi_r-\pi_r]\lambda(\mathcal{E})\mathrm{d}r}\mathrm{d}s.
\end{equation}
Therefore, the optimal control is
\begin{equation*}
  \begin{aligned}
    \hat{u}_{(t,e)}&=C_{5,t}q_t-C_{1,t}\Pi_tE[\hat{X}_t]-C_{1,t}\Sigma_t,\qquad \nu\text{-a.s.}, \\
    \hat{u}_{(t,e)}&=-C_{3,t}\Pi_tE[\hat{X}_t]-C_{3,t}\Sigma_t,\qquad \mu\text{-a.s.}, \\
    \hat{\eta}_i&=-q_{\tau_i}H_{\tau_i},
  \end{aligned}
\end{equation*}
where $q$, $\Pi$ and $\Sigma$ is give by \eqref{qqqt}, \eqref{PPPt} and \eqref{SSSt}, respectively.

Next we present the second example. Because we would like to highlight the characteristics of the jump time, we consider the following forward-backward control system without impulse control,
\begin{equation}\label{1177}
  \left\{
    \begin{aligned}
      \mathrm{d}X_t&=\int_{\mathcal{E}} u_{(t,e)}\lambda(\mathrm{d}e)\mathrm{d}t+\int_{\mathcal{E}} u_{(t,e)}N(\mathrm{d}t, \mathrm{d}e)-\int_{\mathcal{E}}u_{(t,e)}\tilde{N}(\mathrm{d}t,\mathrm{d}e),\qquad X_0=1,\\
      \mathrm{d}Y_t&=-\int_{\mathcal{E}}u_{(t,e)}^2\lambda(\mathrm{d}e)\mathrm{d}t+Z_t\mathrm{d}B_t+\int_{\mathcal{E}} K_t \tilde{N}(\mathrm{d}t,\mathrm{d}e),\qquad Y_1=X_1^2.\\
    \end{aligned}
  \right.
\end{equation}
The cost functional is
\begin{equation}
  \mathcal J(u)=E\bigg[\frac{1}{2}\int_0^1\int_{\mathcal{E}}u^2_{(t,e)}\lambda(\mathrm{d}e)\mathrm{d}t+2\int_0^1\int_{\mathcal{E}} u_{(t,e)}^2 N(\mathrm{d}t,\mathrm{d}e)+\frac{1}{2}X_1^2+\frac{1}{2}Y_0\bigg].
\end{equation}
Firstly, we will show that how to find the optimal control in the progressive
structure. Suppose that $u$ is the optimal control, we have the following equations by Theorem \ref{necessary},
\begin{flalign}
    p_t-2q_tu_{(t,e)}+u_{(t,e)}&=0,\qquad \mu\text{-a.s.},\label{u111}\\
    p_{t-}+4u_{(t,e)}&=0,\qquad \nu\text{-a.s.}.\label{u222}
\end{flalign}
According to \eqref{adjoint}, we obtain $q_t\equiv-\frac{1}{2}$. Then substituting \eqref{u111} into the term of $\mathrm{d}t$ and substituting \eqref{u222} into the terms of $N(\mathrm{d}t,\mathrm{d}e)$ and $\tilde N(\mathrm{d}t,\mathrm{d}e)$, we have
\begin{equation*}
  \begin{aligned}
    X_t&=1-\int_0^t\frac{p_s}{2}\lambda({\mathcal{E}})\mathrm{d}s-\int_0^t\frac{p_{s-}}{4}N(\mathrm{d}s,{\mathcal{E}}) +\int_0^t \frac{p_{s-}}{4}\tilde{N}(\mathrm{d}s,{\mathcal{E}}) =1-\int_0^t\frac{3}{4}p_s\lambda({\mathcal{E}})\mathrm{d}s.
  \end{aligned}
\end{equation*}
The second equation is due to the fact that $p_{t-}$ is predictable
\begin{equation*}
  \int_0^t\frac{p_{s-}}{4}\tilde N(\mathrm{d}s,{\mathcal{E}})=\int_0^t\frac{p_{s-}}{4}N(\mathrm{d}s,{\mathcal{E}}) -\int_0^t\frac{p_{s}}{4}\lambda({\mathcal{E}})\mathrm{d}s.
\end{equation*}
Setting $a:=-\frac{3}{4}\lambda({\mathcal{E}})$, we have
\begin{equation}\label{FBSDE2222}
  \left\{
    \begin{aligned}
      X_t&=1+\int_0^tap_s\mathrm{d}s\\
      p_t&=2X_1-\int_t^T \theta_s\mathrm{d}B_s-\int_t^T\int_{\mathcal{E}}\vartheta_s\tilde{N}(\mathrm{d}s, \mathrm{d}e)
    \end{aligned}
  \right.
\end{equation}
Obviously, $\theta=\vartheta=0$ is the solution of \eqref{FBSDE2222}. Let $p_t=\tilde P_tX_t$ with a continuous differentiable function $\tilde P_t$. Applying It\^o's formula to $p_t$ and comparing the coefficients, we have
\begin{equation}
    \begin{aligned}
      &\dot{\tilde P}_t+a\tilde P_t^2=0,\qquad
      \tilde P_1=2.
    \end{aligned}
\end{equation}
Noting that $a<0$, we obtain the solution of the above ODE is $\tilde P_t=\frac{2}{2at-2a+1}, t\in[0,1]$. Then the optimal control is
\begin{equation}
  u_{(t,e)}=-\frac{\tilde P_t}{2}X_t,\qquad \mu\text{-a.s.};\qquad\text{and}\qquad u_{(t,e)}=-\frac{\tilde P_t}{4}X_{t-},\qquad \nu\text{-a.s.}.
\end{equation}
Now we show that $u$ is an optimal control. For any control $v\in\mathcal{U}_{ad}$, applying It\^o's formula to $P_tx_t^2$, then
\begin{equation*}
  \begin{aligned}
    2X_1^2=P_0&+\int_0^1\dot{\tilde P}_tX_t^2\mathrm{d}t+\int_0^1\int_{\mathcal{E}}2\tilde P_tX_tv_{(t,e)}\lambda(\mathrm{d}e) \mathrm{d}t
    +\int_0^1\int_{\mathcal{E}}2\tilde P_tX_{t-}v_{(t,e)}N(\mathrm{d}t,\mathrm{d}e)\\
    &-\int_0^1\int_{\mathcal{E}} 2\tilde P_tX_{t-}v_{(t,e)} \tilde N(\mathrm{d}t,\mathrm{d}e).
  \end{aligned}
\end{equation*}
By \eqref{1177}, we have
\begin{equation}\label{2266}
  Y_0=X_1^2+\int_0^1\int_{\mathcal{E}}v_{(s,e)}^2\lambda(\mathrm{d}e)\mathrm{d}t-\int_0^1Z_t\mathrm{d}B_t-\int_0^1 \int_{\mathcal{E}} K_t\tilde{N}(\mathrm{d}t,\mathrm{d}e).
\end{equation}
Therefore the cost functional becomes
\begin{equation*}
  \begin{aligned}
    \mathcal J(v)&=E\bigg[\int_0^1\int_{\mathcal{E}}v^2_{(t,e)}\lambda(\mathrm{d}e)\mathrm{d}t+2\int_0^1\int_{\mathcal{E}} v_{(t,e)}^2 N(\mathrm{d}t,\mathrm{d}e)+X_1^2\bigg]\\
    &=\frac{1}{2}\tilde P_0+E\bigg[\int_0^1\int_{\mathcal{E}}\Big\{v^2_{(t,e)}+\frac{1}{2\lambda({\mathcal{E}})}\dot{\tilde P}_tX_t^2+\tilde P_tX_t v_{(t,e)} \Big\}\lambda(\mathrm{d}e)\mathrm{d}t  +\int_0^1\int_{\mathcal{E}}\Big\{2v_{(t,e)}^2+\tilde P_tX_{t-} v_{(t,e)}\Big\} N(\mathrm{d}t,\mathrm{d}e)\bigg]\\
    &=\frac{1}{2}\tilde P_0+E\bigg[\int_0^1\int_{\mathcal{E}}\Big(v_{(t,e)}+\frac{1}{2}\tilde P_tX_t\Big)^2\lambda(\mathrm{d}e) \mathrm{d}t+2\int_0^1\int_{\mathcal{E}}\Big(v_t+\frac{1}{4}\tilde P_tX_{t-}\Big)^2N(\mathrm{d}t,\mathrm{d}e)\bigg]\\
    &\geq \frac{1}{2}\tilde P_0.
  \end{aligned}
\end{equation*}
So $u$ is the optimal control and
\begin{equation*}
  \mathcal J(u)=\frac{1}{2}\tilde P_0=\frac{2}{3\lambda({\mathcal{E}})+2}.
\end{equation*}
Next we will show that how to find the optimal control in the predictable structure. Suppose $u\in\mathcal{U}_{ad}$ is predictable, then the system becomes
\begin{equation*}
  \left\{
    \begin{aligned}
      \mathrm{d}X_t&=2\int_{\mathcal{E}} u_{(t,e)}\lambda(\mathrm{d}e)\mathrm{d}t,\qquad X_0=1,\\
      \mathrm{d}Y_t&=-\int_{\mathcal{E}}u_{(t,e)}^2\lambda(\mathrm{d}e)\mathrm{d}t+Z_t\mathrm{d}B_t+\int_{\mathcal{E}} K_t \tilde{N}(\mathrm{d}t,\mathrm{d}e),\qquad Y_1=X_1^2,\\
    \end{aligned}
  \right.
\end{equation*}
and the cost functional becomes
\begin{equation*}
  \mathcal J(u)=E\bigg[\frac{5}{2}\int_0^1\int_{\mathcal{E}}u^2_{(t,e)}\lambda(\mathrm{d}e)\mathrm{d}t+\frac{1}{2}X_1^2 +\frac{1}{2}Y_0\bigg].
\end{equation*}
Then we know that the optimal control $u$ satisfies the following equation.
\begin{equation*}
  2p_t+6u_{(t,e)}=0.
\end{equation*}
By the same method, we also assume that $p_t=Q_tx_t$ with a continuous differentiable function $Q_t$. Then we have
\begin{equation*}
    \begin{aligned}
      &\dot{\tilde Q}_t-\frac{2\lambda({\mathcal{E}})}{3}\tilde Q_t^2=0,\qquad
      \tilde Q_1=2.
    \end{aligned}
\end{equation*}
We can obtain the solution of the above ODE is $\tilde Q_t=\frac{6}{4(1-t)\lambda({\mathcal{E}})+3}, t\in[0,1]$. For any predictable control $v\in\mathcal{U}_{ad}$, applying It\^o's formula to $\tilde Q_tX_t^2$, we obtain
\begin{equation*}
  2X_1^2=\tilde Q_0+\int_0^1\dot{\tilde Q}_tx_t^2\mathrm{d}t+4\int_0^1\int_{\mathcal{E}}\tilde Q_tx_tv_{(t,e)}\lambda(\mathrm{d}e) \mathrm{d}t.
\end{equation*}
By \eqref{2266}, we have
\begin{equation*}
  \begin{aligned}
    \mathcal J(v)&=E\bigg[\frac{5}{2}\int_0^1\int_{\mathcal{E}}u^2_{(t,e)}\lambda(\mathrm{d}e)\mathrm{d}t+\frac{1}{2}X_1^2 +\frac{1}{2}Y_0\bigg]\\
    &=\frac{1}{2}\tilde Q_0+3E\bigg[\int_0^1\int_{\mathcal{E}}\Big(v_{(t,e)}+\frac{1}{3}\tilde Q_tX_t\Big)^2\lambda(\mathrm{d}e) \mathrm{d}t\bigg]\\
    &\geq \frac{1}{2}\tilde Q_0.
  \end{aligned}
\end{equation*}
So the optimal control and corresponding cost functional are as following
\begin{equation*}
  u_{(t,e)}=-\frac{1}{3}\tilde Q_tX_t,\qquad\text{and}\qquad \mathcal J(u)=\frac{1}{2}\tilde Q_0=\frac{3}{4\lambda({\mathcal{E}})+3}.
\end{equation*}
Obviously, optimal control in progressive structure is better than that in predictable structure.

\begin{remark1}
  From the above example, we can see that it is very important to determine the value of control at the jump time in the progressive structure. On the contrary, due to the predictability of control, its behavior at the jump time can be ignored in the predictable structure.
\end{remark1}

\section{Conclusion}
In this paper, we derive necessary and sufficient stochastic maximum principle for mean-field forward-backward system involving random jumps and impulse control in progressive structure. Our stochastic maximum principle include three parts, the first is the continuous part, the second is the jump part and the last in the impulse part. Compared with the predictable structure, we characterize the characteristics of optimal control at jump time more accurately. Our stochastic maximum principle fully reflects the difference between jump process and continuous Brownian motion. As an application, the related LQ optimal control problem is considered. Moreover, Under some basic assumptions, we give the optimal control and corresponding state feedback representation.

\appendix

\section{Proof of Lemma \ref{SDEsolution-estimate}}
\begin{proof}
  Since the proof of the existence and uniqueness of mean-field SDEP \eqref{SDE} is similar to the Appendix A in \cite{Song2020} and Proposition 4.1 in \cite{Buckdahn2009}, we only give the proof of $L^2$ estimation.

  First, we prove that the estimation holds when $T$ is sufficiently small. Define mappings $\mathscr{T}^i$ as following
  \begin{equation*}
    \begin{aligned}
      \mathscr{T}^i(X^i)_t=x_0^i
      &+\int_0^t\int_{\mathcal{E}} b^i(s,X_s^i,E[X_s^i],e)\lambda(\mathrm{d}e)\mathrm{d}s+\int_0^t \int_{\mathcal{E}} \sigma^i(s,X_s^i,E[X_s^i],e)\lambda(\mathrm{d}e)\mathrm{d}B_s\\
      &+\int_0^t\int_{\mathcal{E}} \gamma^i(s,X_{s-}^i,E[X_{s-}^i],e)N(\mathrm{d}t,\mathrm{d}e) +\int_0^t\int_{\mathcal{E}} c^i(s,X_{s-}^i,E[X_{s-}^i],e)\tilde{N}(\mathrm{d}t,\mathrm{d}e).
    \end{aligned}
  \end{equation*}
  By the same argument of Theorem 2.4 in \cite{Song2021}, we can know that the mappings $\mathscr{T}^i$ is well-defined and is a contraction mapping. Then we have
  \begin{equation*}
    \begin{aligned}
      \|X^1-X^2\|^2&=\|\mathscr{T}^1(X^1)-\mathscr{T}^2(X^2)\|^2=\|\mathscr{T}^1(X^1)-\mathscr{T}^2(X^1) +\mathscr{T}^2(X^1)-\mathscr{T}^2(X^2)\|^2\\
      &\leq 2(\|\mathscr{T}^1(X^1)-\mathscr{T}^2(X^1)\|^2+\|\mathscr{T}^2(X^1)-\mathscr{T}^2(X^2)\|^2)\\
      &\leq 2\|\mathscr{T}^1(X^1)-\mathscr{T}^2(X^1)\|^2+2C(T)\|X^1-X^2\|^2.
    \end{aligned}
  \end{equation*}
  Choosing a sufficiently small $T$ such that $2C(T)<1$, noting that $\|X^1-X^2\|^2<\infty$, we get
  \begin{equation*}
    \|X^1-X^2\|^2\leq \frac{2}{1-2C(T)}\|\mathscr{T}^1(X^1)-\mathscr{T}^2(X^2)\|^2.
  \end{equation*}
  We also obtain
  \begin{equation*}
    \begin{aligned}
      \|\mathscr{T}^1(X^1)-\mathscr{T}^2(X^2)\|^2
      &\leq C|x_0^1-x_0^2|^2+CE\bigg[\bigg(\int_0^T \bigg|\int_{\mathcal{E}} b^1(t,X_t^1,E[X_t^1],e)-b^2(t,X_t^1,E[X_t^1],e)\lambda(\mathrm{d}e)\bigg|\mathrm{d}t \bigg)^2\bigg]\\
      &\quad+CE\bigg[\int_0^T\bigg|\int_{\mathcal{E}}\sigma^1(t,X_t^1,E[X_t^1],e)-\sigma^2(t,X_t^1,E[X_t^1],e) \lambda(\mathrm{d}e)\bigg|^2 \mathrm{d}t\bigg]\\
      &\quad+CE\bigg[\bigg(\int_0^T\int_{\mathcal{E}}\big|\gamma^1(t,X_{t-}^1,E[X_{t-}^1],e)-\gamma^2(t,X_{t-}^1, E[X_{t-}^1],e)\big| N(\mathrm{d}t,\mathrm{d}e)\bigg)^2\bigg]\\
      &\quad+CE\bigg[\int_0^T\int_{\mathcal{E}}\big|c^1(t,X_{t-}^1,E[X_{t-}^1],e)-c^2(t,X_{t-}^1,E[X_{t-}^1],e)\big|^2 N(\mathrm{d}t,\mathrm{d}e)\bigg].
    \end{aligned}
  \end{equation*}
  Then we have the conclusion for sufficiently small time interval. For any $T\in R$, let assumed that $T=2\delta$, where $\delta$ satisfies $2L(\delta)<1$. We give two SDEPs as following
  \begin{equation*}
    \begin{aligned}
      \bar X^i_t&=x_0^i
      +\int_0^t\int_{\mathcal{E}} b^i(s,\bar X_s^i,E[\bar X_s^i],e)\lambda(\mathrm{d}e)\mathrm{d}s+\int_0^t \int_{\mathcal{E}} \sigma^i(s,\bar X_s^i,E[\bar X_s^i],e)\lambda(\mathrm{d}e)\mathrm{d}B_s\\
      &\quad+\int_0^t\int_{\mathcal{E}} \gamma^i(s,\bar X_{s-}^i,E[\bar X_{s-}^i],e)N(\mathrm{d}t,\mathrm{d}e) +\int_0^t\int_{\mathcal{E}} c^i(s,\bar X_{s-}^i,E[\bar X_{s-}^i],e)\tilde{N}(\mathrm{d}t,\mathrm{d}e),\ t\in[0,\delta],
    \end{aligned}
  \end{equation*}
  and
  \begin{equation*}
    \begin{aligned}
      \tilde X^i_t&=\bar X_{\delta}^i
      +\int_0^t\int_{\mathcal{E}} b^i(s,\tilde X_s^i,E[\tilde X_s^i],e)\lambda(\mathrm{d}e)\mathrm{d}s+\int_0^t \int_{\mathcal{E}} \sigma^i(s,\tilde X_s^i,E[\tilde X_s^i],e)\lambda(\mathrm{d}e)\mathrm{d}B_s\\
      &\quad+\int_0^t\int_{\mathcal{E}} \gamma^i(s,\tilde X_{s-}^i,E[\tilde X_{s-}^i],e)N(\mathrm{d}t,\mathrm{d}e) +\int_0^t\int_{\mathcal{E}} c^i(s,\tilde X_{s-}^i,E[\tilde X_{s-}^i],e)\tilde{N}(\mathrm{d}t,\mathrm{d}e),\ t\in [\delta,2\delta].
    \end{aligned}
  \end{equation*}
  It is easy to see that there exist a unique solution of the above equations, respectively. Now we define
  \begin{equation*}
    X_t^i=\left\{
      \begin{aligned}
        &\bar{X}_t^i,\qquad t\in[0,\delta],\\
        &\tilde{X}_t^i,\qquad t\in[\delta,2\delta].
      \end{aligned}
    \right.
  \end{equation*}
  We can check that $X^i$ is the solution of mean-field SDEP \eqref{xiii}. Then we have
  \begin{equation*}
    \begin{aligned}
      E\bigg[\sup_{t\in[0,T]}|X_t^1-X_t^2|^2\bigg]
      &\leq E\bigg[\sup_{t\in[0,\delta]}|X_t^1-X_t^2|^2\bigg] +E\bigg[\sup_{t\in[\delta,2\delta]}|X_t^1-X_t^2|^2\bigg]\\
      &= E\bigg[\sup_{t\in[0,\delta]}|\bar X_t^1-\bar X_t^2|^2\bigg] +E\bigg[\sup_{t\in[\delta,2\delta]}|\tilde X_t^1-\tilde X_t^2|^2\bigg]\\
      &\leq C|x_0^1-x_0^2|^2 +CE\bigg[\bigg(\int_0^T \bigg|\int_{\mathcal{E}} b^1(t,X_t^1,E[X_t^1],e)-b^2(t,X_t^1,E[X_t^1],e)\lambda(\mathrm{d}e)\bigg|\mathrm{d}t \bigg)^2\bigg]\\
      &\quad+CE\bigg[\int_0^T\bigg|\int_{\mathcal{E}}\sigma^1(t,X_t^1,E[X_t^1],e)-\sigma^2(t,X_t^1,E[X_t^1],e) \lambda(\mathrm{d}e)\bigg|^2 \mathrm{d}t\bigg]\\
      &\quad+CE\bigg[\bigg(\int_0^T\int_{\mathcal{E}}\big|\gamma^1(t,X_{t-}^1,E[X_{t-}^1],e)-\gamma^2(t,X_{t-}^1, E[X_{t-}^1],e)\big| N(\mathrm{d}t,\mathrm{d}e)\bigg)^2\bigg]\\
      &\quad+CE\bigg[\int_0^T\int_{\mathcal{E}}\big|c^1(t,X_{t-}^1,E[X_{t-}^1],e)-c^2(t,X_{t-}^1,E[X_{t-}^1],e)\big|^2 N(\mathrm{d}t,\mathrm{d}e)\bigg].
    \end{aligned}
  \end{equation*}
  Then the proof is complete.
\end{proof}

\section{Appendix B}
$(M,\mathscr{M})$ is a measurable space, $K:\Omega\times\mathscr{M}\rightarrow R_+$ is a finite transition kernel from $\Omega$ to $M$.
\begin{proposition1}\label{B.1.}
  Let $(\varphi_n), \varphi:E\times M\rightarrow R$ be $\mathscr{F}\otimes \mathscr{M}$ measurable function,  $(\varphi_n)$ converge to $\varphi$ in $P\times K$. If there exists a $\mathscr{F}\otimes \mathscr{M}$ measurable function $\phi$ satisfying $\int_M\phi(\omega,y)K(\omega, \mathrm{d}y)\leq \infty, P$-a.s. such that $|\varphi_n| \leq \phi$. Then
  \begin{eqnarray*}
    \int_M\varphi_n(\omega,y)K(\omega,\mathrm{d}y)\rightarrow \int_M\varphi(\omega,y)K(\omega,\mathrm{d}y),\qquad \text{in } P.
  \end{eqnarray*}
\end{proposition1}
\begin{proof}
  We define
  \begin{eqnarray*}
    \begin{aligned}
      \eta_n(\omega):&=\int_M\varphi_n(\omega,y)K(\omega,\mathrm{d}y),\\
      \eta(\omega):&=\int_M\varphi(\omega,y)K(\omega,\mathrm{d}y).
    \end{aligned}
  \end{eqnarray*}
  If for any subsequence $\eta_{n_{k}}$ of $\eta_n$, there exists a subsequence $\eta_{n_{k_{l}}}$ of $\eta_{n_{k}}$ converge to $\eta,\ P$-a.s., then we obtain $\eta_n$ converge to $\eta$ in $P$. There exists a subsequence $\varphi_{n_{k}}$ converge to $\varphi$ in $P\times K$, so there exists a subsequence $\varphi_{n_{k_{l}}}$ satisfying $\varphi_{n_{k_{l}}}$ converge to $\varphi$, $P\times K$-a.s..
  Set
  \begin{eqnarray*}
    A:=\{(\omega,y)|\lim_{n_{k_l}\rightarrow \infty}\varphi_{n_{k_l}}(\omega,y)=\varphi(\omega,y),\ P\times K \text{-a.s.}\}.
  \end{eqnarray*}
  For any $\omega\in \Omega$, we have $A_{\omega}:=\{y\in M|(\omega,y)\in A\}$.
  Obviously, $A_{\omega}\in \mathscr{M}$.
  Then
  \begin{eqnarray*}
    \int_\Omega K(\omega,M)P(\mathrm{d}\omega)=(P\times K)(\Omega\times M)=\int_\Omega\bigg(\int_MI_{A_{\omega}}(y) K(\omega,\mathrm{d}y)\bigg)P(\mathrm{d}\omega).
  \end{eqnarray*}
  Thus we have $K(\omega,A_{\omega})=K(\omega,M)$ for every $\omega$, $P$-a.s.. This means that $A_{\omega}^c$ is a zero measurable set in $K(\omega,\cdot)$.
  Setting
  \begin{eqnarray*}
    \begin{aligned}
      B:&=\{\omega\in\Omega|K(\omega,A_{\omega})=K(\omega,M)\},\\
      C:&=\bigg\{\omega\in\Omega|\int_Mg(\omega,y)K(\omega,\mathrm{d}y)\leq \infty\bigg\}.
    \end{aligned}
  \end{eqnarray*}
  It is easy to see that $P(B^c)=P(C^c)=0$. For any $\omega\in B\bigcap C$, $\varphi_{n_{k_l}}(\omega,\cdot)$ converge to $\varphi(\omega,\cdot)$, $K(\omega,\cdot)$-a.s.. Then we have the result by dominated convergence theorem.
\end{proof}

\begin{proposition1}\label{B.2.}
  Under the same assumption of Proposition A1, the different point is that $\phi$ satisfies
  \begin{equation}\label{appendix17}
    \int_{\mathcal{E}}\bigg(\int_M\phi(\omega,y)K(\omega,\mathrm{d}y)\bigg)^2\mu(\mathrm{d}\omega)<\infty,
  \end{equation}
  then we have
  \begin{eqnarray*}
    \int_{\mathcal{E}}\bigg(\int_M\varphi_n(\omega,y)K(\omega,\mathrm{d}y)\bigg)^2\mu(\mathrm{d}\omega)\rightarrow \int_{\mathcal{E}}\bigg(\int_M\varphi(\omega,y)K(\omega,\mathrm{d}y)\bigg)^2\mu(\mathrm{d}\omega).
  \end{eqnarray*}
\end{proposition1}
\begin{proof}
  By \eqref{appendix17}, we have
  \begin{equation*}
    \int_M\phi(\omega,y)K(\omega,\mathrm{d}y)<\infty,\qquad \mu\text{-a.s.}.
  \end{equation*}
  By Proposition \ref{B.1.}, we have
  \begin{eqnarray*}
    \int_M\varphi_n(\omega,y)K(\omega,\mathrm{d}y)\rightarrow \int_M\varphi(\omega,y)K(\omega,\mathrm{d}y)\qquad \text{in } \mu.
  \end{eqnarray*}
  Then we obtain the result by usual dominated convergence theorem.
\end{proof}

\section{ Proof of Lemma \ref{lemma551}}
\begin{proof}
For any $t\geq 0$ and $U\in\mathscr{E}$, we have
\begin{equation*}
  N(\omega,[0,t]\times U)=\sum_{n=1}^{\infty}\mathbbm{1}_{\{U_n\in U\}}\mathbbm{1}_{\{T_n\leq t\}}.
\end{equation*}
Set $X_t:=N(\omega,[0,t]\times U)$. Since $\lambda$ is finite measure, $\{T_n\}_{n\geq 1}$ is a sequence of stopping times which is strictly increasing. Then we have
\begin{equation*}
  X_{T_1}=\mathbbm{1}_{\{U_1\in U\}}.
\end{equation*}
Because for any $U\in\mathscr{B}({\mathcal{E}})$, $X_t:=N(\omega,[0,t]\times U)$ is a progressive process. Then we have $X_{T_1}$ is $\mathscr{F}_{T_1}$ measurable random variable, so $U_1$ is $\mathscr{F}_{T_1}$ measurable random variable. Similarly, we can prove that for any $n\geq 1$, $U_n$ is $\mathscr{F}_{T_n}$ measurable random variable.

For a given $n\geq 1$, we define the following set
\begin{equation*}
  A:=\{(\omega,t,e)\ |\ T_n(\omega)\leq t \text{ and } U_n(\omega)\leq e\}.
\end{equation*}
For a fixed variable $e$, $U_n$ is $\mathscr{F}_{T_n}$ measurable random variable, so $A(e)=\{T_n(\omega)\leq t\}\cap\{U_n(\omega)\leq e\}$ is a progressive set. For fixed $(\omega,t)$, $\mathbbm{1}_{A(\omega,t)}(e)$ is a right continuous function on $R$, so $A$ is $\mathscr{G}\otimes \mathscr{B}({\mathcal{E}})$ measurable. Similarly, we also prove that
\begin{equation*}
  \begin{aligned}
    A_1&:=\{(\omega,t,e)\ |\ T_n(\omega)\leq t \text{ and } U_n(\omega)< e\},\\
    A_2&:=\{(\omega,t,e)\ |\ T_n(\omega)< t \text{ and } U_n(\omega)\leq e\},
  \end{aligned}
\end{equation*}
are $\mathscr{G}\otimes \mathscr{B}({\mathcal{E}})$ measurable. Then we have $[\![(T_n,U_n)]\!]=A-(A_1\cup A_2)$ is $\mathscr{G}\otimes \mathscr{B}({\mathcal{E}})$ measurable.
\end{proof}

\bibliography{sample}
\bibliographystyle{unsrt}

\end{document}